\documentclass{amsart}
\usepackage{amssymb}

\usepackage{amscd}
\usepackage{thmdefs}

\input{tcilatex}
\theoremstyle{definition}
\theoremstyle{remark}
\numberwithin{equation}{section}

\input tcilatex

\begin{document}
\title[Differentiability of quasiconvex functions]{Differentiability of quasiconvex functions on separable Banach spaces}
\author{Patrick J. Rabier}
\address{Department of mathematics, University of Pittsburgh, Pittsburgh, PA 15260}
\email{rabier@imap.pitt.edu}
\maketitle

\begin{abstract}
We investigate the differentiability properties of a real-valued quasiconvex
function $f$ defined on a separable Banach space $X.$ Continuity is only
assumed to hold at the points of a dense subset. If so, this subset is
automatically residual.

Sample results that can be quoted without involving any new concept or
nomenclature are as follows: (i) If $f$ is usc or strictly quasiconvex, then 
$f$ is Hadamard differentiable at the points of a dense subset of $X.$ (ii)
If $f$ is even, then $f$ is continuous and G\^{a}teaux differentiable at the
points of a dense subset of $X.$ In (i) or (ii), the dense subset need not
be residual but, if $X$ is also reflexive, it contains the complement of a
Haar null set. Furthermore, (ii) remains true without the evenness
requirement if the definition of G\^{a}teaux differentiability is
generalized in an unusual, but ultimately natural, way.

The full results are much more general and substantially stronger. In
particular, they incorporate the well known theorem of Crouzeix, to the
effect that every real-valued quasiconvex function on $\Bbb{R}^{N}$ is
Fr\'{e}chet differentiable a.e.
\end{abstract}

\section{Introduction\label{intro}}

According to Rockafellar \cite[p. 428]{Ro70}, it has been known since the
early 20th century that a real-valued convex function on $\Bbb{R}^{N}$ is
everywhere continuous and a.e. Fr\'{e}chet differentiable, although the
geometric leanings prevalent at that time make it difficult to pinpoint the
origin of this result with greater accuracy.

When $\Bbb{R}^{N}$ is replaced by an infinite dimensional Banach space $X,$
a real-valued convex function on $X$ is either continuous at every point or
discontinuous at every point. This makes it obvious that the investigation
of the differentiability properties of convex functions should be limited to
the former class. When $X$ is separable, Mazur \cite{Ma33} proved the
residual G\^{a}teaux differentiability of continuous convex functions in
1933. Recall that, in Baire category terminology, a residual subset is the
complement of a set of first category. Much later, in 1976, Aronszajn \cite
{Ar76} showed that G\^{a}teaux differentiability is also true almost
everywhere, provided that a suitable generalization of null sets is used
when $\dim X=\infty $ (Aronszajn null sets; see Subsection \ref{aronszajn}).

In the theorems of Mazur and Aronszajn, G\^{a}teaux differentiability can be
replaced by Hadamard differentiability -and hence by Fr\'{e}chet
differentiability when $X=\Bbb{R}^{N}$- since both concepts coincide for
locally Lipschitz functions (this is well known \cite[p. 19]{Ya74} and
elementary). More generally, the residual Fr\'{e}chet differentiability of
continuous convex functions was proved in a 1968 landmark paper by Asplund 
\cite{As68} when $X^{*}$ is separable, as well as in other cases that do not
require or imply the separability of $X.$

The dichotomy everywhere/nowhere continuous is no longer true for finite
quasiconvex functions, where the quasiconvexity of $f$ is understood as $%
f(\lambda x+(1-\lambda )y)\leq \max \{f(x),f(y)\}$ for every $x,y$ and every 
$\lambda \in [0,1].$ Nonetheless, in 1981, Crouzeix \cite{Cr81} (see also 
\cite{ChCr87}) proved that, just like convex functions, real-valued
quasiconvex functions on $\Bbb{R}^{N},$ not necessarily continuous, are a.e.
Fr\'{e}chet differentiable. To date, this property has not been extended in
any form to infinite dimensional spaces. In fact, with only a few notable
exceptions, the literature has consistently focused on differentiability
under local Lipschitz continuity conditions, a topic with a long history
which is still the subject of active research; see the recent text \cite
{LiPrTi12}. Of course, local Lipschitz continuity is grossly inadequate to
handle quasiconvex functions, even continuous ones, irrespective of $\dim X.$

It is the purpose of this paper to show how Crouzeix's theorem can be
generalized when $\Bbb{R}^{N}$ is replaced by a separable Banach space $X.$
We provide variants of the theorems of Mazur and Aronszajn in the more
general setting of quasiconvex functions, although neither theorem has a
genuine generalization. In particular, ``mixed'' criteria must be used to
evaluate the size of the set of points of differentiability. We shall return
to this and related issues further below.

As a simple first step, recall that the set of points of discontinuity of
any real-valued function $f$ on a topological space $X$ is an $\mathcal{F}%
_{\sigma }$ (see e.g. \cite[p.78]{HeSt65}, \cite[p. 58]{Ru74}), so that
either this set is of first category, or it has nonempty interior\footnote{%
It is only in Baire spaces (e.g., Banach spaces) that these two options are
mutually exclusive.}. Evidently, no generic differentiability result should
be expected in the latter case, which dictates confining attention to
quasiconvex functions that are continuous at the points of a residual subset
of the (separable) Banach space $X.$ Incidentally, Mazur's theorem breaks
down for this class of functions, even when $X=\Bbb{R}$ and continuity is
assumed. Indeed, in general, a monotone continuous function on $\Bbb{R}$ is
only differentiable at the points of a set of first category (%
\cite[Corollary 1]{Za84}), even though its complement has measure $0$ by
Lebesgue's theorem. As we shall see later, Aronszajn's theorem also fails in
the quasiconvex (even continuous) case when $\dim X=\infty .$

In Banach spaces, residual sets are dense and the converse is trivially true
for $\mathcal{G}_{\delta }$ sets. In particular, a real-valued function is
continuous at the points of a residual subset if and only if it is \textit{\
densely continuous,} i.e.,\textit{\ } continuous at the points of a dense
subset. Common examples of densely continuous functions include the upper or
lower semicontinuous functions (\cite[Lemma 2.1]{Na72}) and, by a theorem
that goes back to Baire himself, the so-called functions of Baire first
class, i.e., the pointwise limits of sequences of continuous functions (%
\cite[p. 12]{Yo80}). An apparently new class (ideally quasiconvex functions)
that contains all the semicontinuous quasiconvex functions -and even all the
quasiconvex functions when $X=\Bbb{R}^{N}$- is described in Subsection \ref
{ideally}.

As shown by the author in \cite{Ra12}, when $X$ is a Baire topological
vector space (tvs), the densely continuous quasiconvex functions $f$ on $X$
have quite simple equivalent characterizations in terms of the lower level
sets 
\begin{equation}
F_{\alpha }:=\{x\in X:f(x)<\alpha \}\text{ and }F_{\alpha }^{\prime
}:=\{x\in X:f(x)\leq \alpha \},  \label{1}
\end{equation}
where $\alpha \in \Bbb{R}\cup \{-\infty \}.$ Of course, since $f$ is
real-valued, $F_{-\infty }=F_{-\infty }^{\prime }=\emptyset ,$ but the
notation will occasionally be convenient. By the quasiconvexity of $f,$ all
the sets $F_{\alpha }$ and $F_{\alpha }^{\prime }$ are convex and,
conversely, if all the $F_{\alpha }$ (or $F_{\alpha }^{\prime }$) are
convex, then $f$ is quasiconvex, which is de Finetti's original definition 
\cite{Fi49}.

In \cite{Ra12} and again in this paper, a special value plays a crucial
role, which is the so-called \textit{\ topological essential infimum }$%
\mathcal{T}\limfunc{ess}\inf_{X}f$ of $f,$ denoted by $m$ for simplicity,
defined by 
\begin{multline}
m:=\mathcal{T}\limfunc{ess}\inf_{X}f:=\sup \{\alpha \in \Bbb{R}:F_{\alpha }%
\text{ is of first category}\}  \label{2} \\
=\inf \{\alpha \in \Bbb{R}:F_{\alpha }\text{ is of second category}\}.
\end{multline}
The last equality in (\ref{2}) follows from the fact that the sets $%
F_{\alpha }$ are linearly ordered by inclusion. Since $f$ is real-valued and 
$X$ is a Baire space, it is always true that $\mathcal{T}\limfunc{ess}
\inf_{X}f\in [-\infty ,\infty ).$ Note also that $\inf_{X}f\leq \mathcal{T}%
\limfunc{ess}\inf_{X}f$ and that the set $F_{m}$ (but not $F_{m}^{\prime }$)
is \emph{always} of first category since it is the union of countably many $%
F_{\alpha }$ with $\alpha <m.$

From now on, $X$ is a separable Banach space and $f:X\rightarrow \Bbb{R}$ is
quasiconvex and densely continuous. Below, we give a synopsis of the
differentiability results proved in this paper. The rough principle is that
the size of the subset $F_{m}$ ranks the amount of differentiability of $f$
(the smaller the better). Even though $F_{m}$ is always of first category,
various refinements will be involved. Many of the concepts and some
technical results needed for the proofs are reviewed or introduced in the
next two sections. In particular, several properties of $F_{m}^{\prime }$
with direct relevance to the differentiability question are established in
Section \ref{residually}.

In Section \ref{above}, we focus on the differentiability of $f$ above level 
$m,$ that is, at the points of $X\backslash F_{m}^{\prime }.$ The special
features of $F_{m}^{\prime }$ make it possible to rely on a theorem of
Borwein and Wang \cite{BoWa05} to prove that $f$ is Hadamard differentiable
on $X\backslash F_{m}^{\prime }$ except at the points of an Aronszajn null
set (Theorem \ref{th9}). If $m=-\infty ,$ so that $X\backslash F_{m}^{\prime
}=X,$ this fully generalizes Aronszajn's theorem and settles the
differentiability issue.

Accordingly, in the remainder of the discussion, $m>-\infty $ is assumed. It
turns out that $X\backslash F_{m}^{\prime }$ is always semi-open, i.e.,
contained in the closure of its interior (which is not true when $%
F_{m}^{\prime }$ is replaced by an arbitrary convex set), so that the
differentiability result just mentioned above still shows that $f$ is
Hadamard differentiable at ``most'' points of $X\backslash F_{m}^{\prime }.$
That $X\backslash F_{m}^{\prime }$ may be empty does not invalidate this
statement. Hence, it only remains to investigate differentiability at the
points of $F_{m}^{\prime }=F_{m}\cup f^{-1}(m).$

It turns out that $f$ is never G\^{a}teaux differentiable at any point of $%
F_{m}$ (Theorem \ref{th11}). Thus, the differentiability of $f$ on $%
F_{m}^{\prime }$ depends solely upon its differentiability at the points of $%
f^{-1}(m)$ and, if $f$ is G\^{a}teaux differentiable at $x\in f^{-1}(m),$
then $Df(x)=0$ (Theorem \ref{th11}), as if $m$ were a genuine minimum. In
summary, the problem is to evaluate the size of those points of $f^{-1}(m)$
at which the directional derivative of $f$ exists and is $0$ in all
directions and to decide whether $f$ is Hadamard differentiable at such
points. While this demonstrates the importance of $m$ in the
differentiability question, the task is not as simple as one might first
hope.

When $X=\Bbb{R}^{N},$ $F_{m}$ is not only of first category but also nowhere
dense because, in finite dimension, \emph{convex} subsets of first category
are nowhere dense. Differentiability when $X$ is separable and $F_{m}$ is
nowhere dense is the topic of Section \ref{differentiability1}. The
technicalities depend upon whether $F_{m}^{\prime }$ is of first or second
category but, ultimately, we show that $f$ is Hadamard -but not Fr\'{e}chet-
differentiable on the complement of the union of a \emph{nowhere dense} set
with an Aronszajn null set (Theorem \ref{th14}). Such unions cannot be
replaced with Aronszajn null sets alone or with sets of first category alone.

Every subset which is the union of a nowhere dense set with an Aronszajn
null set has empty interior (of course, this is false if ``nowhere dense''
is replaced by ``first category'') and the class of such subsets is an
ideal, but not a $\sigma $-ideal. In particular, $f$ above -or any finite
collection of similar functions- is Hadamard differentiable at the points of
a dense subset of $X.$ This is only a ``subgeneric'' differentiability
property but, if $X$ is separable and reflexive, a truly generic result
holds: If $F_{m}$ is nowhere dense, $f$ is Hadamard differentiable on $X$
except at the points of a Haar null set (Theorem \ref{th15}). This should be
related to the rather unexpected fact that, in such spaces, a quasiconvex
function is densely continuous if and only if its set of points of
discontinuity is Haar null (Theorem \ref{th6}).

Examples show that Theorem \ref{th15} is not true for Fr\'{e}chet
differentiability, or if $X$ is not reflexive, or if ``Haar null'' is
replaced by ``Aronszajn null'', even if $f$ is continuous and $X=\ell ^{2}.$
Thus, Aronszajn's theorem cannot be extended to (densely) continuous
quasiconvex functions if $\dim X=\infty .$ Nonetheless, when $X=\Bbb{R}^{N},$
Crouzeix's theorem is recovered because every quasiconvex function on $\Bbb{R%
}^{N}$ is densely continuous, the Haar null sets of $\Bbb{R}^{N}$ are just
the Borel subsets of Lebesgue measure $0$ and Hadamard and Fr\'{e}chet
differentiability coincide on $\Bbb{R}^{N}.$

Aside from $X=\Bbb{R}^{N},$ there are several conditions ensuring that $%
F_{m} $ is nowhere dense. In particular, if $f$ is upper semicontinuous
(usc) or, more generally, if $\inf_{X}f=m,$ for then $F_{m}=\emptyset .$
Also, $F_{m}$ is nowhere dense if $f$ is strongly ideally quasiconvex
(Definition \ref{def1}, Theorem \ref{th7}) or strictly quasiconvex and
densely continuous (Corollary \ref{cor16}).

In general, $F_{m}$ need not be nowhere dense if $\dim X=\infty ,$
especially when $f$ is lower semicontinuous (lsc). Thus, loosely speaking,
the differentiability issue is more delicate for lsc functions than for usc
ones. However, no difficulty arises as a result of passing to the lsc hull,
even though doing so can only enlarge $F_{m}:$ In Subsection \ref{lsc-hulls}%
, we show that Theorems \ref{th14} and \ref{th15} are applicable to a
densely continuous quasiconvex function $f$ if and only if they are
applicable to its lsc hull.

The case when $F_{m}$ is not nowhere dense is discussed in Section \ref
{differentiability2}. This is new territory since it never happens when $X=%
\Bbb{R}^{N}$ or when $f$ is usc, let alone convex and continuous. We show
that a subset $F_{m}^{\ddagger }$ of $X,$ often much smaller than $\overline{%
F}_{m},$ is intimately related to the G\^{a}teaux differentiability
question. Specifically, $F_{m}^{\ddagger }$ consists of all the limits of
the convergent sequences of points of $F_{m}$ that remain in some finite
dimensional subspace of $X.$

In a more arcane terminology, $F_{m}^{\ddagger }$ is the sequential closure
of $F_{m}$ for the finest locally convex topology on $X.$ This type of
closure has been around for quite a while in the literature, but only in
connection with issues far removed from differentiability (topology, moment
problem in real algebraic geometry, etc.) and often in a setting that rules
out infinite dimensional Banach spaces (countable dimension).

In Theorem \ref{th20}, we prove that if $F_{m}^{\ddagger }$ has empty
interior, then $f$ is both continuous and G\^{a}teaux differentiable at the
points of a dense subset of $X.$ Since Hadamard differentiability implies
continuity, this is weaker than the analogous result when $F_{m}$ is nowhere
dense (Theorem \ref{th14}), but applicable in greater generality.

Corollaries are given in which the assumption that $F_{m}^{\ddagger }$ has
empty interior is replaced by a more readily verifiable condition. Most
notably, we show in Corollary \ref{cor23} that every \emph{even} densely
continuous quasiconvex function $f$ on a separable Banach space is
continuous and G\^{a}teaux differentiable at the points of a dense subset of 
$X.$ Furthermore, if $X$ is also reflexive, $f$ is continuous and
G\^{a}teaux differentiable on $X$ except at the points of a Haar null set.
These properties remain true when $f$ exhibits more general symmetries
(Corollary \ref{cor24}).

Section \ref{example} is devoted to an example showing that when $F_{m}$ is
not nowhere dense, the differentiability properties of $f$ at the points of $%
F_{m}^{\prime }$ or, equivalently, $f^{-1}(m),$ are indeed significantly
weaker than when $F_{m}$ is nowhere dense. This confirms that an optimal
outcome cannot be obtained without splitting the investigation of the two
cases.

The results of Section \ref{differentiability2}, particularly those
incorporating some symmetry assumption about $f,$ make it legitimate to ask
whether $F_{m}^{\ddagger }$ always has empty interior. If true, this would
imply that Theorem \ref{th20} is always applicable. In Section \ref{peculiar}%
, we put an early end to this speculation, by exhibiting a convex subset $C$
of $\ell ^{2}$ of first category such that $C^{\ddagger }=\ell ^{2}.$ The
construction of $C$ also shows how to produce densely continuous quasiconvex
functions $f$ such that $m>-\infty $ and $F_{m}=C,$ but we were unable to
determine whether any such function fails to be G\^{a}teaux differentiable
at the points of a dense subset.

The aftermath of $F_{m}^{\ddagger }$ not always having empty interior is
that, when $\dim X=\infty ,$ it remains an open question whether \emph{every}
densely continuous quasiconvex function $f$ on a separable Banach space $X$
is continuous and G\^{a}teaux differentiable at the points of a dense
subset. Nevertheless, in Section \ref{essential}, this question is answered
in the affirmative after the definition of G\^{a}teaux differentiability is
slightly altered.

A basic remark is that the concept of G\^{a}teaux derivative at a point $x$
continues to make sense if it is only required that the directional
derivatives at $x$ exist and are represented by a continuous linear form for
some residual set of directions in the unit sphere. This suffices to define
an ``essential'' G\^{a}teaux derivative at $x$ in a unique way, independent
of the residual set of directions (which is not true if ``residual'' is
replaced by ``dense''). With this extended definition, the problem is
resolved in Theorem \ref{th29}. In particular, if $X$ is reflexive and
separable, every densely continuous quasiconvex function is continuous and
essentially G\^{a}teaux differentiable on $X$ except at the points of a Haar
null set.

\section{Preliminaries\label{preliminaries}}

We collect various definitions and related results that will be used in the
sequel. Most, but not all, of them are known, some more widely than others.
Whenever possible, references rather than proofs are given.

\subsection{Aronszajn null sets\label{aronszajn}}

Let $X$ denote a separable Banach space. If $(\xi _{n})\subset X$ is any
sequence, call $\frak{A}((\xi _{n}))$ the class of all Borel subsets $%
E\subset X$ such that $E=\cup _{n}E_{n}$ where $E_{n}$ is a Borel ``null set
on every line parallel to $\xi _{n}$'', that is, $\lambda _{1}((x+\Bbb{R}\xi
_{n})$ $\cap E_{n})=0$ for every $x\in X,$ where $\lambda _{1}$ is the
one-dimensional Lebesgue measure. The \textit{Aronszajn null} sets (\cite
{Ar76}) are the (Borel) subsets $E$ of $X$ such that $E\in \frak{A}((\xi
_{n}))$ for every sequence $(\xi _{n})$ such that $\overline{\limfunc{span}
\{(\xi _{n})\}}=X.$

The Aronszajn null sets form a $\sigma $-ring that does not contain any
nonempty open subset and every Borel subset of an Aronszajn null set is
Aronszajn null. When $X=\Bbb{R}^{N},$ they are the Borel subsets of Lebesgue
measure $0.$ It was shown by Cs\"{o}rnyei \cite{Cs99} that the Aronszajn
null sets coincide with the \textit{Gaussian null} sets (Phelps \cite{Ph78};
these are the Borel subsets of $X$ that are null for every Gaussian\footnote{%
For every $\ell \in X^{*},$ the measure $\mu _{\ell }$ on $\Bbb{R}$ defined
by $\mu _{\ell }(S):=\mu (\ell ^{-1}(S))$ has a Gaussian distribution.}
probability measure on $X$) and also with the ``cube null'' sets of
Mankiewicz \cite{Ma73}.

Aronszajn null sets are preserved by affine diffeomorphisms of $X,$ i.e.,
translations and bounded invertible linear transformations. As noted by
Matou\u {s}kov\'{a} \cite{Ma97}, if $C$ is a convex subset of $X$ with
nonempty interior, the boundary $\partial C$ is Aronszajn null, but this
need not be true if $\overset{\circ }{C}=\emptyset ,$ even if $C$ is closed
and $X$ is Hilbert.

\subsection{Haar null sets\label{haar}}

Let once again $X$ denote a separable Banach space. The Borel subset $%
E\subset X$ is said to be \textit{Haar null} if there is a Borel probability
measure $\mu $ on $X$ such that $\mu (x+E)=0$ for every $x\in X.$ This
definition is due to Christensen \cite{Ch72}, \cite{Ch74}; see also \cite
{Ma98}. It is obvious that every Aronszajn null set is Haar null, but the
converse is false \cite{Ph78}, except when $X=\Bbb{R}^{N}.$ The Haar null
sets also form a $\sigma $-ring (though this is not obvious from just the
definition) containing no nonempty open subset, every Borel subset of a Haar
null set is Haar null and, just like Aronszajn null sets, they are preserved
by affine diffeomorphisms.

In practice, it is not essential that $\mu $ above be a probability measure.
The definition is unchanged if $\mu $ is a Borel measure having the same
null sets as a probability measure. For example, this happens if $\mu
(E):=\lambda _{1}(\Bbb{R}\xi \cap E)$ where $\xi \in X\backslash \{0\}$ and $%
\lambda _{1}$ is the Lebesgue measure on $\Bbb{R}\xi $ (change $\lambda _{1}$
into $\varphi \lambda _{1}$ with $\varphi (t):=\pi ^{-\frac{1}{2}}e^{-t^{2}}$
).

If, in addition, $X$ is reflexive, every closed convex subset of $X$ with
empty interior is Haar null (Matou\u {s}kov\'{a} \cite{Ma01}) but, as noted
in the previous subsection, not necessarily Aronszajn null.

\subsection{G\^{a}teaux and Hadamard differentiability\label%
{GHdifferentiability}}

A real-valued function $f$ on an open subset $V$ of a Banach space $X$ is
G\^{a}teaux differentiable at $x\in V$ if there is $l_{x}\in X^{*}$ such
that $l_{x}(h)=\lim_{t\rightarrow 0}\frac{f(x+th)-f(x)}{t}$ for every $h\in
X.$ \ If so, $l_{x}$ is denoted by $Df(x).$ If, in addition, the limit is
uniform for $h$ in every compact subset of $X,$ then $f$ is said to be
Hadamard differentiable at $x.$ Equivalently, $f$ is Hadamard differentiable
at $x$ if and only if $Df(x)h=\lim_{t_{n}\rightarrow 0}\frac{
f(x+t_{n}h_{n})-f(x)}{t_{n}}$ for every sequence $(t_{n})\subset \Bbb{R}%
\backslash \{0\}$ tending to $0$ and every sequence $(h_{n})\subset X$
tending to $h$ in $X.$ If $f$ is Hadamard differentiable at $x,$ then it is
continuous at $x.$ This is folklore (see e.g. \cite{AvSm68}) and elementary.
Of course, G\^{a}teaux differentiability alone does not ensure continuity.
When $X=\Bbb{R}^{N},$ it is plain from the first definition that Hadamard
differentiability and Fr\'{e}chet differentiability coincide.

\subsection{Cone monotone functions\label{monotone}}

Let $X$ be a Banach space and $K\subset X$ be a closed convex cone with
nonempty interior. If $V\subset X$ is a nonempty open subset, a function $%
f:V\rightarrow \Bbb{R}$ is said to be $K$\emph{-nondecreasing} if $x\in
V,k\in K\backslash \{0\}$ and $x+k\in V$ imply $f(x)\leq f(x+k).$ This
concept has long proved adequate to extend Lebesgue's theorem on
differentiation of monotone functions. It goes back (in $\Bbb{R}^{2}$) to
the 1937 edition of the book by Saks \cite{Sa64} and resurfaced in the work
of Chabrillac and Crouzeix \cite{ChCr87}. Its use in separable Banach
spaces, by Borwein \textit{et al.} \cite{BoBuLe03} and Borwein and Wang \cite
{BoWa05}, is more recent. The following result is essentially 
\cite[Theorem 18]{BoWa05}.

\begin{theorem}
\label{th1}\emph{(Borwein-Wang)}. Let $X$ be a separable Banach space, $%
K\subset X$ be a closed convex cone with $\overset{\circ }{K}\neq \emptyset $
and $V\subset X$ be a nonempty open subset. Suppose that $f:V\rightarrow 
\Bbb{R}$ is $K$-nondecreasing. Then $f$ is Hadamard differentiable on $V$
except at the points of an Aronszajn null set.
\end{theorem}

In \cite{BoWa05}, Theorem \ref{th1} is proved only when $V=X$ but, as shown
below, it is not difficult to obtain the general case as a corollary. If $%
x,y\in X,$ we use the notation $y\leq _{K}x$ if $x-y\in K$ and $y<_{K}x$ if $%
x-y\in K\backslash \{0\}.$ These relations are obviously transitive since $K$
is stable under addition (but $\leq _{K}$ is an ordering if and only if $%
K\cap (-K)=\{0\};$ this is unimportant).

First, $f$ in Theorem \ref{th1} is locally bounded (above and below) on $V.$
To see this, let $x\in V$ be given and choose $k\in \overset{\circ }{K}.$ In
particular, $k\neq 0.$ After replacing $k$ by $tk$ for $t>0$ small enough,
which does not affect $k\in \overset{\circ }{K},$ it follows from the
openness of $V$ that we may assume that $x-k\in V$ and $x+k\in V.$ If $z\in
X $ and $||z-x||<\varepsilon $ with $\varepsilon >0$ small enough, then $%
z\in V $ and $k+(z-x)\in \overset{\circ }{K},$ so that $x-k<_{K}z,$ and $%
k-(z-x)\in \overset{\circ }{K},$ so that $z<_{K}x+k.$ Therefore, $f(x-k)\leq
f(z)\leq f(x+k),$ which proves the claim.

From the above, for every $x\in V,$ there is an open neighborhood $%
U_{x}\subset V$ of $x$ such that $f$ is bounded on $U_{x}$ and, by the
separability of $X,$ there is a covering of $V$ by countably many $U_{x}.$
Thus, since a countable union of Aronszajn null sets is Aronszajn null and
since $U_{x}\subset V$ implies that $f$ is $K$-nondecreasing on $U_{x},$ it
is not restrictive to prove Theorem \ref{th1} with $V$ replaced by $U_{x}$
or, equivalently, to prove it under the additional assumption that $f$ is
bounded above and below on $V.$ This can be done by using Theorem \ref{th1}
with $V=X$ and $f$ replaced by a finite $K$-nondecreasing extension of $f$
to $X.$ Such an extension can be obtained as follows: Since $f$ is bounded
on $V,$ let $a,b\in \Bbb{R}$ be such that $a\leq f(x)\leq b$ for every $x\in
V$ and set 
\begin{equation*}
\widetilde{f}(x)=\left\{ 
\begin{array}{c}
a\text{ if }x\notin V+K, \\ 
\sup \{f(y):y\in V,y\leq _{K}x\}\text{ if }x\in V+K.
\end{array}
\right. 
\end{equation*}
Since $(V+K)+K=V+K$ and $\leq _{K}$ is transitive, it is straightforward to
check that $\widetilde{f}$ is $K$-nondecreasing on $X,$ that $\widetilde{f}=f
$ on $V$ and that $a\leq \widetilde{f}\leq b,$ so that $\widetilde{f}$ is
real-valued (when $x\in V+K,$ the set $\{y\in V,y\leq _{K}x\}$ is not empty,
so that the supremum in the definition of $\widetilde{f}$ is never $-\infty $%
). Other extensions are described in \cite{BoWa05}, but they need not be
finite when $f$ is finite.

\begin{remark}
The Borwein-Wang theorem was refined by Duda \cite[Remark 5.2]{Du08}: $f$ is
Hadamard differentiable on $V$ except at the points of a subset in a class $%
\widetilde{C}$ introduced by Preiss and Zaj\'{i}\u {c}ek, which is strictly
smaller than the class of Aronszajn null sets \cite[Proposition 13]{PrZa01}.
The class $\widetilde{C}$ may replace the Aronszajn null sets everywhere in
this paper.
\end{remark}

\subsection{Convex sets and category\label{convex sets}}

Convex sets have a few special properties relative to Baire category. Two
useful ones are given below. Like several other results, they are stated in
Banach spaces but are valid in greater generality.

\begin{lemma}
\label{lm2} (\cite[Lemma 3.1]{Ra12}) Let $X$ be a Banach space and let $U$
and $G$ be subsets of $X$ with $U$ open and $G$ convex. If $A\subset X$ is
of first category and $U\backslash A\subset G,$ then $U\subset G.$
\end{lemma}

The next lemma is a by-product of Lemma \ref{lm2}.

\begin{lemma}
\label{lm3} (\cite[Lemma 3.2]{Ra12}) If $X$ is a Banach space and $%
C\varsubsetneq X$ is convex, then $X\backslash C$ is locally of Baire second
category in $X,$ that is, for every open subset $U\subset X$ such that $%
U\cap (X\backslash C)$ is nonempty, $U\cap (X\backslash C)$ is of second
category in $X.$
\end{lemma}

Of course, the spirit of Lemma \ref{lm3} is that even locally, the exterior
of a convex subset $C\neq X$ is always a large set in some sense. While
intuitively clear when $X=\Bbb{R}^{N},$ the existence of dense and convex
proper subsets of infinite dimensional Banach spaces makes this issue less
transparent in general.

\section{Densely continuous quasiconvex functions\label{residually}}

In the Introduction, we defined the densely continuous real valued functions
on a tvs $X$ to be those functions that are continuous at the points of a
dense subset of $X.$ Recall that when $X$ is a Baire space, for instance a
Banach space, this is actually equivalent to continuity at the points of a
residual subset of $X.$

\subsection{Some general properties\label{properties}}

The first theorem of this section is part of the main result of \cite{Ra12}.
It gives a characterization of the densely continuous quasiconvex functions
in terms of their lower level sets. The notation (\ref{1}) is used and will
be used throughout the paper.

\begin{theorem}
\label{th4} (\cite[Corollary 5.3]{Ra12}) Let $X$ be a Banach space and let $%
f:X\rightarrow \Bbb{R}$ be quasiconvex. Set $m:=\mathcal{T}\limfunc{ess}%
\inf_{X}f$ (see (\ref{2})). Then, $f$ is densely continuous if and only if
(i) $\overset{\circ }{F}_{\alpha }\neq \emptyset $ when $\alpha >m$ and (ii)
either $m=-\infty $ or $m>-\infty $ and $F_{\alpha }$ is nowhere dense when $%
\alpha <m.$
\end{theorem}

Other conditions equivalent to, or implying, dense continuity for
quasiconvex functions are given in \cite{Ra12}, but Theorem \ref{th4} will
suffice for our purposes. Since ``first category'', ``nowhere dense'' and
``empty interior'' are synonymous for convex subsets of $\Bbb{R}^{N},$
Theorem \ref{th4} shows that every real-valued quasiconvex function on $\Bbb{%
\ R}^{N}$ is densely continuous. This argument is independent of Crouzeix's
theorem mentioned in the Introduction.

In the next theorem, we prove a number of properties of the set $%
F_{m}^{\prime }$ which will be instrumental in the discussion of the
differentiability question. Part (i) was already noticed in \cite{Ra12}. The
very short proof is given for convenience.

\begin{theorem}
\label{th5} Let $X$ be a Banach space and let $f:X\rightarrow \Bbb{R}$ be
quasiconvex and densely continuous. Set $m:=\mathcal{T}\limfunc{ess}%
\inf_{X}f.$ \newline
(i) If $F_{m}^{\prime }$ (i.e., $f^{-1}(m)$) is of first category, then $%
F_{m}^{\prime }$ (and hence also $f^{-1}(m)$) is nowhere dense.\newline
(ii) If $F_{m}^{\prime }$ (i.e., $f^{-1}(m)$) is of second category, then $%
\overset{\circ }{F_{m}^{\prime }}\neq \emptyset .$ In particular, $\partial %
F_{m}^{\prime }=\overset{\circ }{\partial F_{m}^{\prime }}$ is nowhere
dense. If $X$ is separable, $\partial F_{m}^{\prime }$ is Aronszajn null.%
\newline
(iii) $X\backslash F_{m}^{\prime }$ is semi-open, i.e., contained in the
closure of its interior $X\backslash \overline{F}_{m}^{\prime }.$\newline
(iv) If $X$ is separable, the set of points of discontinuity of $f$ in $%
X\backslash F_{m}^{\prime }$ is contained in an Aronszajn null set.\newline
(v) If $X$ is separable, $F_{m}^{\prime }$ is contained in an Aronszajn null
(Haar null) set if and only if $\overline{F}_{m}^{\prime }$ is Aronszajn
null (Haar null).\newline
\end{theorem}

\begin{proof}
(i) If $x\in \overline{F}_{m}^{\prime }\backslash F_{m}^{\prime },$ then $f$
is not continuous at $x,$ so that $\overline{F}_{m}^{\prime }\subset
F_{m}^{\prime }\cup A$ where $A$ is the set of points of discontinuity of $%
f. $ Since $f$ is densely continuous, $A$ is of first category, whence $%
\overline{F}_{m}^{\prime }$ is of first category, that is, $F_{m}^{\prime }$
is nowhere dense.

(ii) Since $F_{m}^{\prime }$ is of second category, $\overset{\circ }{%
\overline{F}_{m}^{\prime }}\neq \emptyset .$ On the other hand, if $x\in 
\overline{F}_{m}^{\prime }$ and $f$ is continuous at $x,$ then $x\in
F_{m}^{\prime }.$ Since the set $A$ of points of discontinuity of $f$ is of
first category, this means that $\overline{F}_{m}^{\prime }\backslash
A\subset F_{m}^{\prime }.$ In particular, $\overset{\circ }{\overline{F}
_{m}^{\prime }}\backslash A\subset F_{m}^{\prime }$ so that, by Lemma \ref
{lm2}, $\overset{\circ }{\overline{F}_{m}^{\prime }}\subset F_{m}^{\prime }.$
Thus, $\overset{\circ }{\overline{F}_{m}^{\prime }}\subset \overset{\circ }{%
F_{m}^{\prime }}$ and so $\overset{\circ }{F_{m}^{\prime }}=\overset{\circ }{%
\overline{F}_{m}^{\prime }}\neq \emptyset .$

A convex set with nonempty interior has the same boundary as its closure (%
\cite[p. 105]{Be74}) and the boundary of a closed set is nowhere dense.
Thus, $\partial F_{m}^{\prime }$ is nowhere dense. Since $\overset{\circ }{%
F_{m}^{\prime }}\neq \emptyset $ and $F_{m}^{\prime }$ is convex, $\partial
F_{m}^{\prime }$ is Aronszajn null when $X$ is separable (Subsection \ref
{aronszajn}).

(iii) When $\dim X=\infty ,$ this does not follow from the convexity of $%
F_{m}^{\prime }$ alone (the complement of a dense convex subset is not
semi-open) but, by (i) and (ii), we also know that either $F_{m}^{\prime }$
is nowhere dense, or $\overset{\circ }{F_{m}^{\prime }}\neq \emptyset .$

If $F_{m}^{\prime }$ is nowhere dense, $X\backslash \overline{F}_{m}^{\prime
}$ is dense in $X$ and therefore in $X\backslash F_{m}^{\prime }.$ Suppose
now that $\overset{\circ }{F_{m}^{\prime }}\neq \emptyset .$ The only points
of $X\backslash F_{m}^{\prime }$ which are not already in $X\backslash 
\overline{F}_{m}^{\prime }$ lie on $\partial (X\backslash F_{m}^{\prime
})=\partial F_{m}^{\prime }.$ As just pointed out in the proof of (ii), $%
F_{m}^{\prime }$ and $\overline{F}_{m}^{\prime }$ have the same boundary
because $F_{m}^{\prime }$ is convex and $\overset{\circ }{F_{m}^{\prime }}
\neq \emptyset .$ Therefore, if $x\in X\backslash F_{m}^{\prime }$ and $%
x\notin X\backslash \overline{F}_{m}^{\prime },$ then $x\in \partial 
\overline{F}_{m}^{\prime }=\partial (X\backslash \overline{F}_{m}^{\prime })$
is in the closure of $X\backslash \overline{F}_{m}^{\prime }.$

(iv) As before, let $A$ denote the set of points of discontinuity of $f.$
The claim is that $A\backslash F_{m}^{\prime }$ is contained in an Aronszajn
null set. If $x\in A\backslash F_{m}^{\prime },$ then $f(x)>m$ and there are
a sequence $(x_{n})\subset X$ with $\lim_{n\rightarrow \infty }x_{n}=x$ and $%
\alpha \in \Bbb{R}$ such that either $f(x_{n})<\alpha <f(x)$ for every $n\in 
\Bbb{N}$ or $f(x)<\alpha <f(x_{n})$ for every $n\in \Bbb{N}.$ With no loss
of generality, we may assume that $\alpha \in \Bbb{Q}$ and, since $f(x)>m,$
that $\alpha >m$ (of course, this is redundant when $f(x)<\alpha $).

In the first case (i.e., $f(x_{n})<\alpha <f(x)$), $x\in \overline{F}
_{\alpha }$ but $x\notin F_{\alpha },$ so that $x\in \partial F_{\alpha }.$
In the second (i.e., $f(x)<\alpha <f(x_{n})$), $x\in F_{\alpha }$ but $%
x\notin \overset{\circ }{F_{\alpha }},$ so that once again $x\in \partial
F_{\alpha }.$ Thus, $A\backslash F_{m}^{\prime }\subset \cup _{\alpha \in 
\Bbb{Q},\alpha >m}\partial F_{\alpha }.$ By Theorem \ref{th4}, $\overset{
\circ }{F_{\alpha }}\neq \emptyset $ when $\alpha >m,$ so that $\partial
F_{\alpha }$ is Aronszajn null since $F_{\alpha }$ is convex (Subsection \ref
{aronszajn}) and so $\cup _{\alpha \in \Bbb{Q},\alpha >m}\partial F_{\alpha
} $ is Aronszajn null.

(v) The sufficiency is trivial. To prove the necessity, suppose then that $%
F_{m}^{\prime }$ is contained in an Aronszajn null (Haar null) set. As noted
in the proof of (i) above, $\overline{F}_{m}^{\prime }\subset F_{m}^{\prime
}\cup A$ where $A$ is the set of points of discontinuity of $f$ and so $%
\overline{F}_{m}^{\prime }\subset F_{m}^{\prime }\cup (A\backslash
F_{m}^{\prime }).$ By (iv), $A\backslash F_{m}^{\prime }$ is contained in an
Aronszajn null set, so that $\overline{F}_{m}^{\prime }$ is Aronszajn null
(Haar null) since it is Borel.
\end{proof}

While the quasiconvexity of $f$ is actually unnecessary in part (i) of
Theorem \ref{th5}, it is important in part (ii). By Theorem \ref{th4}, every
lower level set $F_{\alpha }^{\prime }$ -not just $F_{m}^{\prime }$- is
either nowhere dense or has nonempty interior, but it is only when $\alpha
=m $ that both options are possible ($F_{\alpha }^{\prime }$ is nowhere
dense when $\alpha <m$ since $F_{\alpha }^{\prime }\subset F_{\beta }$ with $%
\alpha <\beta <m$ and has nonempty interior when $\alpha >m$). In contrast, $%
F_{m}$ is always of first category but it may or may not be nowhere dense.

\begin{remark}
\label{rm1}It is readily checked that the complement of an Aronszajn null
set in a semi-open subset is dense in that subset. Thus, when $X$ is
separable, parts (iii) and (iv) of Theorem \ref{th5} show that the set of
points of continuity of $f$ in $X\backslash F_{m}^{\prime }$ is dense in $%
X\backslash F_{m}^{\prime }.$
\end{remark}

In a separable and reflexive Banach space, a finite quasiconvex function is
densely continuous if and only if its set of points of discontinuity is Haar
null. This perhaps surprising result is proved next.

\begin{theorem}
\label{th6}Let $X$ be a reflexive and separable Banach space. The
quasiconvex function $f:X\rightarrow \Bbb{R}$ is densely continuous if and
only if the set of points of discontinuity of $f$ is Haar null.
\end{theorem}

\begin{proof}
Since the complement of a Haar null set is dense, the ``if'' part is clear.
Conversely, since the set $A$ of points of discontinuity of $f$ is an $%
\mathcal{F}_{\sigma }$ (hence Borel), it suffices to show that if $f$ is
densely continuous, $A$ is contained in a Haar null set. This will be seen
by a suitable modification of the proof of part (iv) of Theorem \ref{th5}.

If $x\in A,$ there are a sequence $(x_{n})\subset X$ with $%
\lim_{n\rightarrow \infty }x_{n}=x$ and $\alpha \in \Bbb{R}$ such that
either $f(x_{n})<\alpha <f(x)$ for every $n\in \Bbb{N}$ or $f(x)<\alpha
<f(x_{n})$ for every $n\in \Bbb{N}.$ With no loss of generality, we may
assume that $\alpha \in \Bbb{Q}$ and $\alpha \neq m:=\mathcal{T}\limfunc{ess}
\inf_{X}f<\infty .$

In the first case, $x\in \overline{F}_{\alpha }$ but $x\notin F_{\alpha },$
so that $x\in \partial F_{\alpha }$ and, in the second, $x\in F_{\alpha }$
but $x\notin \overset{\circ }{F_{\alpha }},$ so that once again $x\in
\partial F_{\alpha }.$ If ($m>-\infty $ and) $\alpha <m,$ it follows from
Theorem \ref{th4} that $\partial F_{\alpha }=\overline{F}_{\alpha }$ has
empty interior. Since $X$ is reflexive and separable, $\partial F_{\alpha }$
is Haar null (Subsection \ref{haar}). If $\alpha >m,$ then $\overset{\circ }{%
F_{\alpha }}\neq \emptyset $ by Theorem \ref{th4}, so that $\partial
F_{\alpha }$ is Aronszajn null (Subsection \ref{aronszajn}) and therefore
Haar null. Thus, $\cup _{\alpha \in \Bbb{Q}\backslash \{m\}}\partial
F_{\alpha }$ is Haar null and $A\subset \cup _{\alpha \in \Bbb{Q}\backslash
\{m\}}\partial F_{\alpha }.$ Since $A$ is Borel -indeed an $\mathcal{F}%
_{\sigma }$- it is Haar null.
\end{proof}

When $X=\Bbb{R}^{N},$ the set of points of discontinuity of $f$ is even $%
\sigma $-porous; see \cite[Theorem 19]{BoWa05} (as pointed out in 
\cite[Remark 5.2]{Ra12}, the lsc assumption in that theorem is not needed).

\subsection{Ideally quasiconvex functions\label{ideally}}

The subset $C$ of the Banach space $X$ is said to be ideally convex if $%
\Sigma _{n=1}^{\infty }\lambda _{n}x_{n}\in C$ for every bounded sequence $%
(x_{n})\subset C$ and every sequence $(\lambda _{n})\subset [0,1]$ such that 
$\Sigma _{n=1}^{\infty }\lambda _{n}=1.$ The boundedness of $(x_{n})$
ensures that $\Sigma _{n=1}^{\infty }\lambda _{n}x_{n}$ is absolutely
convergent. Such subsets were introduced by Lif\v {s}ic \cite{Li70} in 1970
but, apparently, they have not been used in connection with quasiconvex
functions. Indeed, without Theorem \ref{th4}, the purpose of doing so is not
apparent.

\begin{definition}
\label{def1} The function $f:X\rightarrow \Bbb{R}$ is ideally quasiconvex
(strongly ideally quasiconvex) if its lower level sets $F_{\alpha }^{\prime }
$ ($F_{\alpha }$) are ideally convex.
\end{definition}

It is readily seen that a strongly ideally quasiconvex function is ideally
quasiconvex (use $F_{\alpha }^{\prime }=\cap _{\beta >\alpha }F_{\beta }$),
but the converse is false; see Remark \ref{rm2} later. Also, $f$ is ideally
quasiconvex if and only if $f\left( \Sigma _{n=1}^{\infty }\lambda
_{n}x_{n}\right) \leq \sup_{n}f(x_{n})$ for every bounded sequence $%
(x_{n})\subset X$ and every sequence $(\lambda _{n})\subset [0,1]$ such that 
$\Sigma _{n=1}^{\infty }\lambda _{n}=1.$ Strong ideal convexity amounts to $%
f\left( \Sigma _{n=1}^{\infty }\lambda _{n}x_{n}\right) <\alpha $ whenever $%
f(x_{n})<\alpha $ for every $n.$ This has nothing to do with strict
quasiconvexity. For instance, every usc quasiconvex function is strongly
ideally quasiconvex (see below).

If $C\subset X$ is ideally convex, it is CS-closed as defined by Jameson 
\cite{Ja72} and the two concepts coincide when $C$ is bounded. By a
straightforward generalization of Carath\'{e}odory's theorem (see for
example \cite{CoWe72}), every convex subset of $\Bbb{R}^{N}$ is ideally
convex. Therefore, every quasiconvex function on $\Bbb{R}^{N}$ is strongly
ideally quasiconvex. Open convex subsets are CS-closed (\cite{FrTa79}),
hence ideally convex, and the same thing is trivially true of closed convex
subsets. In particular, lsc (usc) quasiconvex functions are ideally
quasiconvex (strongly ideally quasiconvex).

We shall also need the remark that an ideally convex subset $C$ can only be
nowhere dense or have nonempty interior. Indeed, if $C$ is not nowhere
dense, $C\cap B$ is dense in $B$ for some nonempty open ball $B.$ Since $C$
and $B$ are ideally convex, $C\cap B$ is ideally convex (obvious) and
bounded and therefore CS-closed. By \cite[Corollary 1]{Ja72}, this implies
that $C\cap B=B,$ so that $B\subset C$ and so $C$ has nonempty interior.

The main properties of (strongly) ideally quasiconvex functions are captured
in the next theorem.

\begin{theorem}
\label{th7}Let $X$ be a Banach space. \newline
(i) If $f:X\rightarrow \Bbb{R}$ is ideally quasiconvex, then $f$ is
quasiconvex and densely continuous.\newline
(ii) If $f:X\rightarrow \Bbb{R}$ is strongly ideally quasiconvex and $m:=%
\mathcal{T}\limfunc{ess}\inf_{X}f,$ then $f$ is quasiconvex and densely
continuous and $F_{m}$ is nowhere dense. More generally, this is true if $f$
is ideally quasiconvex and $F_{m}$ is ideally convex.
\end{theorem}

\begin{proof}
(i) If $\alpha >m,$ then $F_{\alpha }^{\prime }$ is ideally convex but not
nowhere dense since $F_{\alpha }^{\prime }\supset F_{\alpha }$ is of second
category by definition of $m.$ Thus $\overset{\circ }{F_{\alpha }^{\prime }}%
\neq \emptyset $ by the remark before the theorem. Now, let $\alpha <m$ and,
by contradiction, suppose that $F_{\alpha }$ is not nowhere dense, so that
the larger $F_{\alpha }^{\prime }$ (ideally convex) is not nowhere dense
either. By the remark before the theorem, $\overset{\circ }{F_{\alpha
}^{\prime }}\neq \emptyset $ and so $\overset{\circ }{F_{\beta }}\neq
\emptyset $ for every $\beta >\alpha .$ But $F_{\beta }$ is of first
category by definition of $m$ if $\beta \in (\alpha ,m)$ and a contradiction
arises. Thus, $F_{\alpha }$ is nowhere dense when $\alpha <m.$ That $f$ is
densely continuous now follows from Theorem \ref{th4}.

(ii) By (i), $f$ is densely continuous. Since $F_{m}$ is ideally convex and
has empty interior (it is of first category), it is nowhere dense, once
again by the remark before the theorem.
\end{proof}

Strongly ideally quasiconvex functions are the closest generalization of
finite dimensional quasiconvex functions. As we shall see in Section \ref
{differentiability1}, both have the same differentiability properties. By
Theorem \ref{th7} (i), ideally quasiconvex functions possess only the
continuity properties of the finite dimensional case.

\section{Differentiability above the essential infimum\label{above}}

Throughout this section, $X$ is a separable Banach space and $f:X\rightarrow 
\Bbb{R}$ is quasiconvex and densely continuous. With $m:=\mathcal{T}\limfunc{
ess}\inf_{X}f,$ the goal is to prove that $f$ is Hadamard differentiable at
every point of $X\backslash F_{m}^{\prime }$ except at the points of an
Aronszajn null set.

The last preliminary lemma will enable us to use Theorem \ref{th1} to settle
the differentiability question at points of $X\backslash F_{m}^{\prime }.$
The line of argument of the proof, but with other assumptions, has been used
before (\cite[Theorem 3.1]{Cr05}, \cite[Proposition 2]{BoWa05}).

\begin{lemma}
\label{lm8} Let $x\in X\backslash F_{m}^{\prime }$ be a point of continuity
of $f.$ Then, $x\in X\backslash \overline{F}_{m}^{\prime }$ and there is an
open neighborhood $U_{x}$ of $x$ contained in $X\backslash \overline{F}%
_{m}^{\prime }$ and a closed convex cone $K_{x}\subset X$ with nonempty
interior such that $f$ is $K_{x}$-nondecreasing on $U_{x}.$
\end{lemma}

\begin{proof}
Since $f(x)>m,$ choose $\alpha \in (m,f(x)),$ so that $\overset{\circ }{F}%
_{\alpha }\neq \emptyset $ by Theorem \ref{th4}. Pick $x_{-}\in \overset{%
\circ }{F}_{\alpha }$and $\varepsilon >0$ small enough that $\overline{B}%
(x_{-},2\varepsilon )\subset F_{\alpha },$ that $U_{x}:=B\left(
x,\varepsilon \right) \subset X\backslash F_{\alpha }^{\prime }$ (this is
possible since $f(x)>\alpha $ and $f$ is continuous at $x;$ in particular, $%
U_{x}$ is contained in the interior $X\backslash \overline{F}_{m}^{\prime }$
of $X\backslash F_{m}^{\prime }$) and that $0\notin \overline{B}%
(2(x-x_{-}),2\varepsilon ).$ Now, set $x_{+}:=2x-x_{-},$ so that $x$ is the
midpoint of $x_{-}$ and $x_{+}$ and 
\begin{equation*}
K_{x}:=\cup _{\lambda \geq 0}\lambda \overline{B}(x_{+}-x_{-},2\varepsilon
)=\cup _{\lambda \geq 0}\lambda \overline{B}(2(x-x_{-}),2\varepsilon ).
\end{equation*}
Clearly, $K_{x}$ is a convex cone with nonempty interior. That it is closed
easily follows from the assumption $0\notin \overline{B}(2(x-x_{-}),2%
\varepsilon ).$ Also, if $k\in K_{x}\backslash \{0\},$ then $x-\frac{k}{%
2\lambda }\in \overline{B}\left( x_{-},\varepsilon \right) $ where $\lambda
>0$ is chosen such that $k\in \lambda \overline{B}(x_{+}-x_{-},2\varepsilon
).$ Hence, if $y\in U_{x}=B\left( x,\varepsilon \right) ,$ it follows that $%
y-\frac{k}{2\lambda }=(y-x)+(x-\frac{k}{2\lambda })\in B(x_{-},2\varepsilon
).$

Suppose now that $y\in U_{x},k\in K_{x}\backslash \{0\}$ and $y+k\in U_{x}.$
From the above, $y-\frac{k}{2\lambda }=z\in B(x_{-},2\varepsilon ).$ By
writing $\frac{1}{2\lambda }=\frac{t}{1-t}$ for some $t\in (0,1),$ this
yields $y=t(y+k)+(1-t)z.$ Since $f$ is quasiconvex, $f(y)\leq \max
\{f(y+k),f(z)\}=f(y+k),$ the latter because $f(z)<\alpha <f(y+k)$ (recall $%
z\in F_{\alpha }$ and $y+k\in U_{x}\subset X\backslash F_{\alpha }^{\prime }$%
).
\end{proof}

\begin{theorem}
\label{th9}Let $X$ be a separable Banach space and let $f:X\rightarrow \Bbb{R%
}$ be quasiconvex and densely continuous. Set $m:=\mathcal{T}\limfunc{ess}%
\inf_{X}f.$ Then, $f$ is Hadamard differentiable on $X\backslash
F_{m}^{\prime }$ except at the points of an Aronszajn null set.
\end{theorem}

\begin{proof}
With no loss of generality, assume that $X\backslash F_{m}^{\prime }\neq
\emptyset .$ Since $X\backslash \overline{F}_{m}^{\prime }$ is dense in $%
X\backslash F_{m}^{\prime }$ (Theorem \ref{th5} (iii)), it follows that $%
X\backslash \overline{F}_{m}^{\prime }\neq \emptyset .$ If $\mathcal{C}_{m}$
denotes the set of points of continuity of $f$ in $X\backslash F_{m}^{\prime
},$ then $\mathcal{C}_{m}\subset X\backslash \overline{F}_{m}^{\prime }$ by
Lemma \ref{lm8} and $\mathcal{C}_{m}\neq \emptyset $ since $X\backslash 
\overline{F}_{m}^{\prime }$ is open and nonempty and $f$ is densely
continuous.

If $x\in \mathcal{C}_{m},$ it follows from Lemma \ref{lm8} and Theorem \ref
{th1} that there is an open neighborhood $U_{x}$ of $x$ contained in $%
X\backslash \overline{F}_{m}^{\prime }$ such that $f$ is Hadamard
differentiable on $U_{x}$ except at the points of an Aronszajn null set.
Since $X$ is separable, the open set $U:=\cup _{x\in \mathcal{C}%
_{m}}U_{x}\subset X\backslash \overline{F}_{m}^{\prime }$ is Lindel\"{o}f
and therefore coincides with the union of countably many $U_{x}.$ As a
result, $f$ is Hadamard differentiable on $U$ except at the points of an
Aronszajn null set.

Lastly, since $U$ contains all the points of continuity of $f$ in $%
X\backslash F_{m}^{\prime },$ the points of $(X\backslash F_{m}^{\prime
})\backslash U$ are the points of discontinuity of $f$ in $X\backslash
F_{m}^{\prime }.$ By Theorem \ref{th5} (iv), these points are contained in
an Aronszajn null set, so that, as claimed, $f$ is Hadamard differentiable
on $X\backslash F_{m}^{\prime }$ except at the points of an Aronszajn null
set.
\end{proof}

A variant of Remark \ref{rm1} may be repeated: Since $X\backslash
F_{m}^{\prime }$ is semi-open (Theorem \ref{th5} (iii)), Theorem \ref{th9}
implies that $f$ is Hadamard differentiable at the points of a dense subset
of $X\backslash F_{m}^{\prime }.$ The next corollary is obvious and settles
the case when $m=-\infty .$

\begin{corollary}
\label{cor10}Let $X$ be a separable Banach space and let $f:X\rightarrow 
\Bbb{R}$ be quasiconvex and densely continuous. If $\mathcal{T}\limfunc{ess}%
\inf_{X}f=-\infty ,$ then $f$ is Hadamard differentiable on $X$ except at
the points of an Aronszajn null set.
\end{corollary}

\section{Hadamard differentiability when $F_{m}$ is nowhere dense\label%
{differentiability1}}

Since the case when $m:=\mathcal{T}\limfunc{ess}\inf_{X}f=-\infty $ was
settled in Corollary \ref{cor10}, it is henceforth assumed that $m>-\infty .$
To avoid giving the impression that this is a restrictive assumption, this
is not recorded in the statement of the results, which are all true when $%
m=-\infty $ (by Corollary \ref{cor10}, or because their assumptions
implicitly rule out $m=-\infty $).

Theorem \ref{th9} answers the differentiability question at points of $%
X\backslash F_{m}^{\prime },$ but it remains to investigate the
differentiability at points of $F_{m}^{\prime }.$ To begin with, we show
that if $x\in F_{m}^{\prime },$ the G\^{a}teaux derivative $Df(x)$ can only
exist if $x\in f^{-1}(m)$ and, if so, that $Df(x)=0.$ In particular, this
shows that when it comes to differentiability, $m$ has the same property as
a pointwise infimum.

\begin{theorem}
\label{th11}Let $X$ be a Banach space and let $f:X\rightarrow \Bbb{R}$ be
quasiconvex.\newline
(i) If $x\in F_{m},$ then $f$ is not G\^{a}teaux differentiable at $x.$%
\newline
(ii) If $x\in f^{-1}(m)$ and $f$ is G\^{a}teaux differentiable at $x,$ the
G\^{a}teaux derivative $Df(x)$ is $0.$
\end{theorem}

\begin{proof}
(i) Since G\^{a}teaux differentiability does not imply continuity, this does
not follow from the fact that $f$ is not continuous at any point of $F_{m}$ (%
\cite[Remark 4.1]{Ra12}).

With no loss of generality, assume that $x=0\in F_{m}.$ By the convexity of $%
F_{m},$ the cone $CF_{m}:=\cup _{\lambda >0}\lambda F_{m}$ is the countable
union $CF_{m}=\cup _{n\in \Bbb{N}}\,nF_{m}$ and hence of first category.
Thus, $X\backslash CF_{m}$ is nonempty and, if $h\in X\backslash CF_{m},$
then $h\neq 0$ (since $0\in F_{m}\subset CF_{m}$) and $th\notin F_{m}$ for
every $t>0.$ As a result, $f(th)\geq m$ when $t>0.$ Since $f(0)<m,$ it
follows that $\lim_{t\rightarrow 0^{+}}\frac{f(th)-f(0)}{t}=\infty ,$ so
that $f$ is not G\^{a}teaux differentiable at $0.$

(ii) Once again, by translation, we may assume that $x=0$ (but now $0\notin
F_{m}$ since $f(0)=m$). By contradiction, assume that $Df(0)\neq 0,$ so that 
$\Sigma :=\{h\in X:Df(0)h<0\}$ is an open half-space. If $h\in \Sigma ,$
then $f(th)<f(0)=m$ for $t>0$ small enough, say $0<t<t_{h}.$ Thus, $%
G:=\{th:h\in \Sigma ,0<t<t_{h}\}\subset F_{m}$ and so $G$ is of first
category. On the other hand, if $h\in \Sigma ,$ then $n^{-1}h\in G$ if $n\in 
\Bbb{\ N}$ and $n^{-1}<t_{h}.$ In other words, $\Sigma \subset \cup _{n\in 
\Bbb{N}}\,nG.$ Since $G$ is of first category, $\Sigma $ is of first
category. Therefore, it cannot be an open half-space and a contradiction is
reached.
\end{proof}

The proof of part (ii) of Theorem \ref{th11} does not use the quasiconvexity
of $f.$ At any rate, Theorem \ref{th11} makes the differentiability of $f$
on $F_{m}^{\prime }$ look like a rather straightforward problem:\ It
suffices to check whether the directional derivative of $f$ at $x\in
f^{-1}(m)$ exists and is $0$ in every direction. However, the evaluation of
the size of the set of such points is not that simple. The remainder of this
section deals with the case when $F_{m}$ is nowhere dense.

If $F_{m}^{\prime }$ is of first category, then $F_{m}^{\prime }$ is nowhere
dense (Theorem \ref{th5} (i)). In particular, $F_{m}$ is nowhere dense and
there is no need to investigate the differentiability of $f$ at the points
of $f^{-1}(m)$ to prove the ``subgeneric'' Hadamard differentiability of $f:$

\begin{theorem}
\label{th12}Let $X$ be a separable Banach space and let $f:X\rightarrow \Bbb{%
R}$ be quasiconvex and densely continuous. With $m:=\mathcal{T}\limfunc{ess}%
\inf_{X}f,$ suppose that $F_{m}^{\prime }$ (i.e., $f^{-1}(m)$) is of first
category. Then, $F_{m}^{\prime }$ is nowhere dense and $f$ is Hadamard
differentiable on the open and dense subset $X\backslash \overline{F}%
_{m}^{\prime }$ except at the points of an Aronszajn null set. \newline
In particular, if $F_{m}^{\prime }$ is contained in an Aronszajn null (Haar
null) set, then $f$ is Hadamard differentiable on $X$ except at the points
of an Aronszajn null (Haar null) set.
\end{theorem}

\begin{proof}
As just recalled above, $F_{m}^{\prime }$ is nowhere dense. Thus, $%
X\backslash \overline{F}_{m}^{\prime }$ is open and dense in $X$ and it
suffices to use Theorem \ref{th9}.

If $F_{m}^{\prime }$ is contained in an Aronszajn null (Haar null) set, then 
$\overline{F}_{m}^{\prime }$ is Aronszajn null (Haar null) by Theorem \ref
{th5} (v) and so it has empty interior. Thus, $F_{m}^{\prime }$ is of first
category and the first part applies. That $X\backslash F_{m}^{\prime }$ may
be replaced by $X$ in the first part is obvious.
\end{proof}

We now prove a similar result when $F_{m}^{\prime }$ is of second category.

\begin{theorem}
\label{th13}Let $X$ be a separable Banach space and let $f:X\rightarrow \Bbb{%
R}$ be quasiconvex and densely continuous. With $m:=\mathcal{T}\limfunc{ess}%
\inf_{X}f,$ suppose that $F_{m}^{\prime }$ (i.e., $f^{-1}(m)$) is of second
category and that $F_{m}$ is nowhere dense. Then, $f$ is Hadamard
differentiable on the open and dense subset $X\backslash \overline{F}_{m}$
except at the points of an Aronszajn null set. \newline
In particular, if $\overline{F}_{m}$ is Aronszajn null (Haar null), then $f$
is Hadamard differentiable on $X,$ except at the points of an Aronszajn null
(Haar null) set.
\end{theorem}

\begin{proof}
By Theorem \ref{th5} (ii), $\overset{\circ }{F_{m}^{\prime }}\neq \emptyset $
and $\partial F_{m}^{\prime }$ is Aronszajn null. On the other hand, since $%
\overset{\circ }{F_{m}^{\prime }}\neq \emptyset $ and $F_{m}$ is nowhere
dense, $\overset{\circ }{F_{m}^{\prime }}\backslash \overline{F}_{m}$ is
open and dense in $\overset{\circ }{F_{m}^{\prime }}.$ Furthermore, since $%
f=m$ on $F_{m}^{\prime }\backslash F_{m}$ and hence on $\overset{\circ }{%
F_{m}^{\prime }}\backslash \overline{F}_{m},$ it follows that $f$ is
Fr\'{e}chet differentiable with derivative $0$ on $\overset{\circ }{
F_{m}^{\prime }}\backslash \overline{F}_{m}.$ Accordingly, the points of $%
F_{m}^{\prime }\backslash \overline{F}_{m}$ where $f$ is not Hadamard
differentiable are contained in the Aronszajn null set $\partial
F_{m}^{\prime }.$ By Theorem \ref{th9}, the points of $X\backslash
F_{m}^{\prime }$ where $f$ is not Hadamard differentiable are also contained
in an Aronszajn null set and the conclusion follows from $X\backslash 
\overline{F}_{m}\subset (X\backslash F_{m}^{\prime })\cup (F_{m}^{\prime
}\backslash \overline{F}_{m}).$

If $\overline{F}_{m}$ is Aronszajn null (Haar null), it has empty interior,
so that $F_{m}$ is nowhere dense and the first part applies. That $%
X\backslash \overline{F}_{m}$ may be replaced by $X$ in the first part is
obvious.
\end{proof}

By combining Theorem \ref{th12} and Theorem \ref{th13}, we find

\begin{theorem}
\label{th14}Let $X$ be a separable Banach space and let $f:X\rightarrow \Bbb{%
R}$ be quasiconvex and densely continuous. With that $m:=\mathcal{T}\limfunc{
ess}\inf_{X}f,$ suppose that $F_{m}$ is nowhere dense. Then, $f$ is Hadamard
differentiable on an open and dense subset of $X$ except at the points of an
Aronszajn null set. \newline
In particular, if $F_{m}^{\prime }$ is contained in an Aronszajn null (Haar
null) set, or if $F_{m}^{\prime }$ is of second category and $\overline{F}%
_{m}$ is Aronszajn null (Haar null), then $f$ is Hadamard differentiable on $%
X$ except at the points of an Aronszajn null (Haar null) set.
\end{theorem}

The union of two subsets of $X,$ one of which is nowhere dense -not just of
first category- and the other Aronszajn null, has empty interior (its
complement is dense). Every finite -but not countable- union of such sets is
of the same type. Thus, Theorem \ref{th14} not only ensures that $f$ is
Hadamard differentiable on a dense subset of $X$ but also that finitely many
functions satisfying its hypotheses are simultaneously Hadamard
differentiable on the same dense subset of $X.$ Since this need not be true
for countably many functions, this property can only be called subgeneric
(but see Theorem \ref{th15} below).

\begin{remark}
Phelps' example \cite[p. 80]{Ph89} of a G\^{a}teaux differentiable norm on $%
\ell ^{1}$ which is nowhere Fr\'{e}chet differentiable shows that Hadamard
differentiability cannot be replaced by Fr\'{e}chet differentiability in
Theorem \ref{th14}.
\end{remark}

If $f$ is convex and continuous, it follows from \cite{Ar76} that $f$ is
Hadamard differentiable except at the points of an Aronszajn null set. By
Theorem \ref{th14}, this is still true when $f$ is only quasiconvex and
continuous if $F_{m}^{\prime }$ is of second category (because $%
F_{m}=\emptyset $), but false if $F_{m}^{\prime }$ (closed) is of first
category and not Aronszajn null. For instance, \cite[Example 6]{BoGo09} is a
counter-example with $X$ Hilbert. However, when $X$ is reflexive, Theorem 
\ref{th14} still yields a full generalization of Crouzeix's theorem:

\begin{theorem}
\label{th15}Let $X$ be a reflexive and separable Banach space and let $%
f:X\rightarrow \Bbb{R}$ be quasiconvex and densely continuous. With $m:=%
\mathcal{T}\limfunc{ess}\inf_{X}f,$ suppose that $F_{m}$ is nowhere dense.
Then, $f$ is Hadamard differentiable on $X$ except at the points of a Haar
null set.
\end{theorem}

\begin{proof}
Recall that in reflexive separable Banach spaces, closed convex subsets with
empty interior are Haar null (Subsection \ref{haar}). In particular, since $%
F_{m}$ is nowhere dense, $\overline{F}_{m}$ is Haar null. If $F_{m}^{\prime
} $ is of second category, the result follows directly from Theorem \ref
{th14}. On the other hand, if $F_{m}^{\prime }$ is of first category, then $%
F_{m}^{\prime }$ is actually nowhere dense by Theorem \ref{th5} (i). Thus, $%
\overline{F}_{m}^{\prime }$ is Haar null and Theorem \ref{th14} is once
again applicable.
\end{proof}

Theorem \ref{th15} should be put in the perspective of Theorem \ref{th6}.
When $X=\Bbb{R}^{N},$ every quasiconvex function is densely continuous, $%
F_{m}$ is always nowhere dense and Haar null sets have Lebesgue measure $0,$
so that Crouzeix's theorem is recovered.

Theorem \ref{th15} is not true if $X$ is not reflexive: Every nonreflexive
Banach space contains a closed convex subset $C$ with empty interior which
is not Haar null (\cite{Ma97}). If $f:=\chi _{X\backslash C},$ then $f$ is
quasiconvex and continuous on $X\backslash C$ (open and dense in $X$) and
hence densely continuous. Also, $m=1$ and $F_{1}=C.$ Thus, $f$ is not even
G\^{a}teaux differentiable at any point of $C$ by Theorem \ref{th11}.

\begin{remark}
Just like Theorem \ref{th14}, Theorem \ref{th15} already fails in the convex
case if Hadamard differentiability is replaced by Fr\'{e}chet
differentiability: In \cite{MaMa99}, an example is given of an equivalent
norm on $\ell ^{2}$ which is Fr\'{e}chet differentiable only at the points
of an Aronszajn null set. Obviously, no such set is the complement of a Haar
null set.
\end{remark}

\subsection{Special cases\label{special}}

There are a number of special cases of Theorem \ref{th14} and Theorem \ref
{th15}. In particular, by Theorem \ref{th7}, they are always applicable when 
$f$ is strongly ideally quasiconvex (e.g., usc) or just ideally convex with $%
F_{m}$ ideally convex. More generally, the dense continuity assumption is
satisfied if $f$ is ideally convex (Theorem \ref{th7}) or, as pointed out in
the Introduction, if $f$ is the pointwise limit of continuous functions.
Dense continuity can also be deduced from Theorem \ref{th4} or other
criteria in \cite{Ra12}.

The condition that $F_{m}$ is nowhere dense holds trivially if $\inf_{X}f=m$
(in particular, if $m=-\infty $), for then $F_{m}=\emptyset .$ This happens
when $f$ is usc, but in other cases as well; see for instance Theorem \ref
{th17}. As shown below, $F_{m}$ is nowhere dense -and more- when $f$ is
strictly quasiconvex. Recall that this means $f(\lambda x+(1-\lambda
)y)<\max \{f(x),f(y)\}$ whenever $x\neq y$ and $\lambda \in (0,1).$ It is
not hard to check that $f$ is strictly quasiconvex if and only if it is
quasiconvex and $f(\lambda x+(1-\lambda )y)<f(x)$ whenever $x\neq
y,f(x)=f(y) $ and $\lambda \in (0,1).$ This is also equivalent to saying
that $f$ is quasiconvex and nonconstant on any nontrivial line segment (%
\cite[Theorem 9]{DiAvZa81}).

\begin{corollary}
\label{cor16}Let $X$ be a separable Banach space and let $f:X\rightarrow 
\Bbb{R}$ be strictly quasiconvex and densely continuous. Set $m:=\mathcal{T}%
\limfunc{ess}\inf_{X}f.$ Then, $F_{m}^{\prime }$ is nowhere dense and $f$ is
Hadamard differentiable on a dense open subset of $X$ except at the points
of an Aronszajn null set. Furthermore:\newline
(i) If $F_{m}^{\prime }$ is contained in an Aronszajn null (Haar null) set,
then $f$ is Hadamard differentiable on $X$ except at the points of an
Aronszajn null (Haar null) set.\newline
(ii) If $X$ is reflexive, $f$ is Hadamard differentiable on $X$ except at
the points of a Haar null set.
\end{corollary}

\begin{proof}
By Theorem \ref{th5} (i), $F_{m}^{\prime }$ is nowhere dense if (and only
if) $f^{-1}(m)$ is of first category. If so, $F_{m}\subset F_{m}^{\prime }$
is also nowhere dense, so that everything follows once again from Theorem 
\ref{th14} and Theorem \ref{th15}.

To complete the proof, we show that $f^{-1}(m)$ is of first category. If $%
F_{m}=\emptyset ,$ the strict quasiconvexity of $f$ shows that $%
f^{-1}(m)=F_{m}^{\prime }$ contains at most one point. If $F_{m}\neq
\emptyset ,$ assume with no loss of generality that $0\in F_{m}.$ If $x\in
f^{-1}(m)$ then $f\left( \frac{1}{2}x\right) <\max \{f(0),f(x)\}=m$ and so $%
\frac{1}{2}x\in F_{m}.$ As a result, $f^{-1}(m)\subset 2F_{m}.$ Since $F_{m}$
is of first category, the same thing is true of $2F_{m}$ and hence of $%
f^{-1}(m).$
\end{proof}

From the above proof, Corollary \ref{cor16} is still true if $%
f_{|F_{m}^{\prime }}$ is strictly quasiconvex. Unlike ideal quasiconvexity,
strict quasiconvexity does not imply dense continuity and must be checked
separately. (The sum of a continuous strictly convex function and a
real-valued nowhere continuous convex function is strictly convex, hence
strictly quasiconvex, but nowhere continuous.)

We now show that (sub)generic Hadamard differentiability can always be
obtained after a rather mild modification of the function $f$ that does not
affect the Hadamard derivative of $f$ at any point where such a derivative
exists.

\begin{theorem}
\label{th17}Let $X$ be a separable Banach space and let $f:X\rightarrow \Bbb{%
R}$ be quasiconvex and densely continuous. Suppose that $m:=\mathcal{T}%
\limfunc{ess}\inf_{X}f>-\infty $ and set $g:=\max \{f,m\}.$ The following
properties hold:\newline
(i) $g$ is quasiconvex and densely continuous and $f=g$ in the residual set $%
X\backslash F_{m}.$\newline
(ii) If $f$ is Hadamard differentiable at $x,$ then $g$ is Hadamard
differentiable at $x$ and $Df(x)=Dg(x).$\newline
(iii) $g$ is Hadamard differentiable on a dense open subset of $X$ except at
the points of an Aronszajn null set. Furthermore:\newline
(iii-{\small 1}) If $F_{m}^{\prime }$ is of first category and contained in
an Aronszajn null (Haar null) set, or if $F_{m}^{\prime }$ is of second
category, then $g$ is Hadamard differentiable on $X$ except at the points of
an Aronszajn null (Haar null) set.\newline
(iii-{\small 2}) If $X$ is reflexive, $g$ is Hadamard differentiable on $X$
except at the points of a Haar null set.
\end{theorem}

\begin{proof}
(i) is obvious since every point of continuity of $f$ is a point of
continuity of $g.$ To prove (ii), assume that $f$ is Hadamard differentiable
at $x.$ If $x\notin F_{m}^{\prime },$ then $f$ is continuous at $x$ and so $%
f $ and $g$ coincide on a neighborhood of $x.$ Therefore, $g$ is Hadamard
differentiable at $x$ and $Dg(x)=Df(x).$ If $x\in F_{m}^{\prime },$ it
follows from Theorem \ref{th11} that $f(x)=m$ and that $Df(x)=0.$ Thus, $%
\lim_{t\rightarrow 0}\frac{f(x+th)-f(x)}{t}=0$ uniformly for $h$ in the
compact subsets of $X.$ Also, $g(x)=m$ and $g(x+th)=f(x+th)$ if $x+th\in
F_{m}^{\prime },$ whereas $g(x+th)=m$ if $x+th\in F_{m}.$ As a result, $%
\frac{|g(x+th)-g(x)|}{t}\leq \frac{|f(x+th)-f(x)|}{t}$ and so $%
\lim_{t\rightarrow 0}\frac{g(x+th)-g(x)}{t}=0$ uniformly for $h$ in the
compact subsets of $X.$ This shows that $g$ is Hadamard differentiable at $x$
and that $Dg(x)=0.$

(iii) Clearly, $\mathcal{T}\limfunc{ess}\inf_{X}g=m,$ $G_{m}:=\{x\in
X:g(x)<m\}=\emptyset $ and $G_{m}^{\prime }:=\{x\in X:g(x)\leq
m\}=F_{m}^{\prime },$ so that it suffices to use Theorem \ref{th14} and
Theorem \ref{th15} with $f$ replaced by $g.$
\end{proof}

\subsection{Differentiability of lsc hulls\label{lsc-hulls}}

Earlier, we established that every usc quasiconvex function $f$ is Hadamard
differentiable on a large subset because (a) $f$ is densely continuous and
(b) $F_{m}=\emptyset .$ In contrast, if $f$ is lsc, (a) is still true but
(b) holds only if $\inf_{X}f=m.$ Worse, $F_{m}$ need not be nowhere dense
(see Remark \ref{rm2}). However, as we shall see, this cannot happen by just
passing from a function $f$ to its lsc hull $\underset{-}{f}$ (largest lsc
function majorized by $f$). More specifically, if $f$ is quasiconvex and
densely continuous, we shall see that Theorem \ref{th14} and Theorem \ref
{th15} are simultaneously applicable to $f$ and $\underset{-}{f}$.

Recall that $\underset{-}{f}(x):=\lim \inf_{x^{\prime }\rightarrow
x}f(x^{\prime })=\inf \{\alpha \in \Bbb{R}:x\in \overline{F}_{\alpha }\}.$
Either characterization shows that if $f$ is quasiconvex, then $\underset{-}{
f}$ is quasiconvex, but $\underset{-}{f}$ may achieve the value $-\infty .$
Accordingly, $x\in X$ will be called a point of continuity of $\underset{-}{f%
}$ if $\underset{-}{f}(x)\in \Bbb{R}$ and $\underset{-}{f}$ is continuous at 
$x.$ Equivalently, a point of discontinuity $x$ of $\underset{-}{f}$ is a
point where $\underset{-}{f}(x)=-\infty ,$ or $\underset{-}{f}(x)\in \Bbb{R}$
but $\underset{-}{f}$ is not continuous at $x.$

\begin{lemma}
\label{lm18}Let $f:X\rightarrow \Bbb{R}$ be quasiconvex and densely
continuous and let $\underline{f}$ denote its lsc hull. Set $m:=\mathcal{T}%
\limfunc{ess}\inf_{X}f$ and $G_{m}:=\{x\in X:\underset{-}{f}%
(x)<m\},G_{m}^{\prime }:=\{x\in X:\underset{-}{f}(x)\leq m\}.$ The following
properties hold:\newline
(i) $f$ and $\underset{-}{f}$ have the same points of continuity and $%
\underset{-}{f}=f$ at such points \newline
(ii) $m=\mathcal{T}\limfunc{ess}\inf_{X}\underset{-}{f}.$\newline
(iii) $\overline{G}_{m}=\overline{F}_{m}.$\newline
(iv) $F_{m}^{\prime }\subset G_{m}^{\prime }\subset F_{m}^{\prime }\cup
Z_{m}^{\prime },$ where $Z_{m}^{\prime }$ is of first category and contained
in an Aronszajn null set. In particular, $G_{m}^{\prime }$ is of first
category (contained in an Aronszajn null set, contained in a Haar null set)
if and only if the same thing is true of $F_{m}^{\prime }.$
\end{lemma}

\begin{proof}
(i) is proved in \cite[Theorem 6.4]{Ra12}.

(ii) Since $f$ is densely continuous, its set of points of discontinuity is
of first category. Thus, by (i), $\underset{-}{f}=f$ except on a set of
first category. This proves (ii) since it is readily checked that modifying
a function on a set of first category does not affect its topological
essential infimum.

(iii) is trivial when $m=-\infty $ since $G_{-\infty }=F_{-\infty
}=\emptyset ,$ so assume $m>-\infty .$ First, $F_{m}\subset G_{m}$ since $%
\underset{-}{f}\leq f.$ Next, $G_{m}\subset \overline{F}_{m},$ for if $x\in
G_{m},$ there is a sequence $(x_{n})$ tending to $x$ such that $\lim
f(x_{n})=\underset{-}{f}(x)<m.$ (Alternatively, notice $G_{m}=\cup _{\alpha
<m}\overline{F}_{\alpha }.$)

(iv) That $F_{m}^{\prime }\subset G_{m}^{\prime }$ follows from $\underset{-%
}{f}\leq f.$ Suppose now that $x\in Z_{m}^{\prime }:=G_{m}^{\prime
}\backslash F_{m}^{\prime }.$ Then, $x$ is a point of discontinuity of $f,$
for otherwise $\underset{-}{f}(x)=f(x)$ by (i). This already shows that $%
Z_{m}:=G_{m}^{\prime }\backslash F_{m}^{\prime }$ is of first category. In
addition, by Theorem \ref{th9}, the set of points of discontinuity of $f$ in 
$X\backslash F_{m}^{\prime }$ -and therefore $Z_{m}^{\prime }$- is contained
in an Aronszajn null set. The ``in particular'' part is then obvious.
\end{proof}

From Lemma \ref{lm18}, if $f$ is quasiconvex and densely continuous, the
size requirements about $F_{m}$ and $F_{m}^{\prime }$ in Theorem \ref{th14}
or Theorem \ref{th15} are satisfied if and only if they are satisfied by $%
G_{m}$ and $G_{m}^{\prime }.$ Therefore, these theorems are applicable to $f$
and $\underset{-}{f}$ as soon as they are applicable to either of them (if $%
\underset{-}{f}$ is not real-valued, first replace $f$ by $\arctan f$). If
so, both are Hadamard differentiable on the same dense subset of $X.$ In
addition, whenever $f$ and $\underset{-}{f}$ are Hadamard differentiable at
the same point $x,$ then $D\underset{-}{f}(x)=Df(x).$ Indeed, by Lemma \ref
{lm18} (i), $\underset{-}{f}(x)=f(x),$ so that, since $\underset{-}{f}\leq
f, $%
\begin{equation*}
D\underset{-}{f}(x)h=\lim_{t\rightarrow 0^{+}}\frac{\underset{-}{f}(x+th)-%
\underset{-}{f}(x)}{t}\leq \lim_{t\rightarrow 0^{+}}\frac{f(x+th)-f(x)}{t}
=Df(x)h,
\end{equation*}
for every $h\in X,$ whence $D\underset{-}{f}(x)h=Df(x)h$ upon changing $h$
into $-h.$

When $\underset{-}{f}$ is replaced by the usc hull $\overset{-}{f}$ of $f,$
there is no close relative of Lemma \ref{lm18}, even though $f$ and $%
\overset{-}{f}$ coincide away from a set of first category; see 
\cite[Lemma 5.1]{Ra12}, where it is also shown that $\overset{-}{f}$ is
finite. In particular, the ``optimal'' differentiability properties of $%
\overset{-}{f}$ do not provide any useful information about the
differentiability properties of $f.$

\section{G\^{a}teaux differentiability on a dense subset\label%
{differentiability2}}

We now investigate the differentiability properties of densely continuous
quasiconvex functions without making the additional assumption that $F_{m}$
is nowhere dense. Recall that $F_{m}$ is always of first category and may
fail to be nowhere dense only when $m>-\infty $ and $\dim X=\infty .$
Several special cases when $F_{m}$ is nowhere dense were discussed in
Section \ref{differentiability1}.

If $X$ is a vector space and $C\subset X,$ we set 
\begin{equation}
C^{\ddagger }:=\cup _{\dim Z<\infty }\overline{C\cap Z},  \label{3}
\end{equation}
where $Z$ denotes a vector subspace of $X$ and the closure $\overline{C\cap Z%
}$ is relative to the natural euclidean topology of $Z.$ This means that $%
C^{\ddagger }$ is the sequential closure of $C$ in the finest locally convex
topology on $X$ (\cite[p. 56]{Sc86}) or in any finer topology (such as the
finite topology \cite{KaKl63}), although such characterizations will not
play a role here. Nonetheless, it is informative to point out that, in
general, $(C^{\ddagger })^{\ddagger }\neq C^{\ddagger }$ (see a counter
example in \cite{CiMaNe11}), so that $C^{\ddagger }$ is not the closure of $%
C $ for any topology on $X.$

In this section, the only important (though trivial) features are that $%
C\subset C^{\ddagger }$ and $C^{\ddagger }\subset \overline{C}$ if $X$ is a
normed space, that $C^{\ddagger }$ is convex if $C$ is convex, that $%
Y^{\ddagger }=Y$ for \emph{every} affine subspace $Y$ of $X,$ plus the
following alternative characterization of $X\backslash C^{\ddagger }$ when $%
C $ is convex.

\begin{lemma}
\label{lm19}Let $X$ be a vector space and $C\subset X$ be a convex subset.
Then, $x\in X\backslash C^{\ddagger }$ if and only if, for every $h\in
X\backslash \{0\},$ there is $t_{h}>0$ such that $(x-t_{h}h,x+t_{h}h)\cap
C=\emptyset .$
\end{lemma}

\begin{proof}
Suppose first that $x\notin C^{\ddagger },$ so that $x\notin C.$ It suffices
to prove that if $h\in X\backslash \{0\},$ then $x+th\notin C$ for $t>0$
small enough. Indeed, since $x\notin C,$ the same property with $h$ replaced
by $-h$ yields the existence of $t_{h}.$ Now, if $x+t_{n}h\in C$ for some
positive sequence $(t_{n}),$ it is obvious that $x\in \overline{C\cap Z}$
with $Z:=\limfunc{span}\{x,h\},$ so that $x\in C^{\ddagger },$ which is a
contradiction.

Suppose now that for every $h\in X\backslash \{0\},$ there is $t_{h}>0$ such
that $(x-t_{h}h,x+t_{h}h)\cap C=\emptyset .$ In particular, $x\notin C.$ By
contradiction, if $x\in C^{\ddagger },$ there is a finite dimensional
subspace $Z$ such that $x\in \overline{C\cap Z}.$ If so, $C\cap Z\neq
\emptyset ,$ so that the relative interior of $C\cap Z$ is also nonempty. If 
$z$ is any point of this relative interior, then $(x,z]\subset C\cap Z$ (%
\cite[p. 45]{Ro70}). Thus, with $h:=z-x\neq 0,$ it follows that $%
(x-th,x+th)\cap C\supset (x,x+th)\neq \emptyset $ for every $t>0$ and a
contradiction is reached.
\end{proof}

It follows from Lemma \ref{lm19} that if $C$ is convex, then $x\in
C^{\ddagger }$ if and only if there are $h\in X\backslash \{0\}$ and $%
t_{h}>0 $ such that $(x,x+t_{h}h]\subset C.$ (In particular, when $C$ is
convex, $C^{\ddagger }$ is unchanged if $\dim Z\leq 2$ in (\ref{3}).)

From Lemma \ref{lm19}, it is hardly surprising that $C^{\ddagger }$ should
have something to do with G\^{a}teaux differentiability. The relationship is
fully clarified in the proof of the next theorem which, in a weaker form, is
a generalization, of Theorems \ref{th14} and \ref{th15}.

\begin{theorem}
\label{th20}Let $X$ be a separable Banach space and let $f:X\rightarrow \Bbb{%
R}$ be quasiconvex and densely continuous. With $m:=\mathcal{T}\limfunc{ess}%
\inf_{X}f,$ suppose that $F_{m}^{\ddagger }$ has empty interior. Then, $f$
is continuous and G\^{a}teaux differentiable on a dense subset of $X.$ 
\newline
If $X$ is reflexive and $F_{m}^{\ddagger }$ is contained in a Haar null set,
then $f$ is continuous and G\^{a}teaux differentiable on $X$ except at the
points of a Haar null set.
\end{theorem}

\begin{proof}
If $F_{m}^{\prime }$ is of first category, then $F_{m}^{\prime }$ is nowhere
dense by Theorem \ref{th5} (i). Thus, $F_{m}$ is nowhere dense and a
stronger result follows from Theorems \ref{th14} and \ref{th15}. Thus, in
the remainder of the proof, we assume that $F_{m}^{\prime }$ is of second
category. In particular, $m>-\infty .$

Since $F_{m}^{\prime }$ is of second category, $\overset{\circ }{%
F_{m}^{\prime }}\neq \emptyset $ by Theorem \ref{th5} (ii). Let $x\in 
\overset{\circ }{F_{m}^{\prime }}\backslash F_{m}^{\ddagger }$ and $h\in
X\backslash \{0\}$ be given. By Lemma \ref{lm19}, there is $t_{h}>0$ such
that $(x-t_{h}h,x+t_{h}h)\subset X\backslash F_{m}.$ After shrinking $t_{h}$
if necessary, we may assume that $(x-t_{h}h,x+t_{h}h)\subset \overset{\circ 
}{F_{m}^{\prime }}\backslash F_{m}^{\ddagger }.$ Since $\overset{\circ }{%
F_{m}^{\prime }}\backslash F_{m}^{\ddagger }\subset F_{m}^{\prime
}\backslash F_{m}\subset f^{-1}(m),$ it follows that $f=m$ on $%
(x-t_{h}h,x+t_{h}h),$ so that the derivative of $f$ at $x$ in the direction $%
h$ exists and is $0.$ Since $h\in X\backslash \{0\}$ is arbitrary, $f$ is
G\^{a}teaux differentiable at $x$ with $Df(x)=0.$ Thus, $f$ is G\^{a}teaux
differentiable at every point of $\overset{\circ }{F_{m}^{\prime }}%
\backslash F_{m}^{\ddagger }.$

Now, let $x\in \overset{\circ }{F_{m}^{\prime }}$ be given and let $r>0$ be
such that $B(x,r)\subset \overset{\circ }{F_{m}^{\prime }}.$ Since $%
F_{m}^{\ddagger }$ has empty interior, $\overset{\circ }{F_{m}^{\prime }}
\backslash F_{m}^{\ddagger }$ is dense in $\overset{\circ }{F_{m}^{\prime }}%
, $ so that there is $y\in B(x,r)\cap \left( \overset{\circ }{F_{m}^{\prime }%
}\backslash F_{m}^{\ddagger }\right) .$ Let $\varepsilon >0$ be such that $%
B(y,\varepsilon )\subset B(x,r).$ Since $y\notin F_{m}^{\ddagger }$ and $%
F_{m}^{\ddagger }$ is convex, $B(y,\varepsilon )\backslash F_{m}^{\ddagger }$
is of second category by Lemma \ref{lm3}, whereas the set of points of
discontinuity of $f$ is only of first category. As a result, $%
B(y,\varepsilon )\backslash F_{m}^{\ddagger }$ contains a point of
continuity $z$ of $f.$ Therefore, $z\in B(x,r)$ and $f$ is both continuous
and G\^{a}teaux differentiable at $z.$

From the above, $f$ is continuous and G\^{a}teaux differentiable at the
points of a dense subset of $\overset{\circ }{F_{m}^{\prime }}.$ By
convexity, $\overset{\circ }{F_{m}^{\prime }}\neq \emptyset $ is dense in $%
F_{m}^{\prime },$ so that the same subset is also dense in $F_{m}^{\prime }.$
On the other hand, since Hadamard differentiability implies continuity, $f$
is continuous and G\^{a}teaux differentiable at the points of a dense subset
of $X\backslash F_{m}^{\prime }$ by Theorem \ref{th9}. Altogether, this
proves that $f$ is continuous and G\^{a}teaux differentiable at the points
of a dense subset of $X.$

To complete the proof, assume that $X$ is reflexive and that $%
F_{m}^{\ddagger }$ is contained in a Haar null set. In particular, $%
F_{m}^{\ddagger }$ has empty interior. Thus, from the above and Theorem \ref
{th9}, the set of points where $f$ is not G\^{a}teaux differentiable is
contained in a Haar null set. Indeed, such points can only be in $\overset{
\circ }{F_{m}^{\prime }}\cap F_{m}^{\ddagger }$ (contained in a Haar null
set), or in $\partial F_{m}^{\prime }$ (Aronszajn null since $\overset{\circ 
}{F_{m}^{\prime }}\neq \emptyset $) or in $X\backslash F_{m}^{\prime }$ and
contained in an Aronszajn null set.

By Theorem \ref{th6}, the set of points of discontinuity of $f$ is also Haar
null and so the set of points where $f$ is not continuous or not G\^{a}teaux
differentiable is Haar null.
\end{proof}

The counter example to Theorem \ref{th15} given in the previous section also
shows that the reflexivity of $X$ cannot be omitted in the second part of
Theorem \ref{th20}.

\begin{remark}
Theorem \ref{th20} or any of its corollaries below is a statement about the
differentiability properties of $f$ on the entire space $X.$ In particular,
it does not record the fact that a stronger property is true at the points
of $X\backslash F_{m}^{\prime }$ (Theorem \ref{th9}; recall that $%
X\backslash F_{m}^{\prime }$ is always semi-open).
\end{remark}

The set $F_{m}^{\ddagger }$ is often much smaller than $\overline{F}_{m}.$
In the next four corollaries, $F_{m}^{\ddagger }$ (but not necessarily $%
\overline{F}_{m}$) has empty interior, although this is not explicitly
assumed.

\begin{corollary}
\label{cor21}Let $X$ be a separable Banach space and let $f:X\rightarrow 
\Bbb{R}$ be quasiconvex and densely continuous. With $m:=\mathcal{T}\limfunc{
ess}\inf_{X}f,$ suppose that for every finite dimensional subspace $Z$ of $X,
$ there is $\alpha <m$ such that $F_{m}\cap Z\subset \overline{F}_{\alpha }.$
Then, $f$ is continuous and G\^{a}teaux differentiable on a dense subset of $%
X.$ \newline
If $X$ is reflexive, $f$ is continuous and G\^{a}teaux differentiable on $X$
except at the points of a Haar null set.
\end{corollary}

\begin{proof}
By Theorem \ref{th20}, it suffices to show that $F_{m}^{\ddagger }$ has
empty interior and that it is contained in a Haar null set if $X$ is
reflexive.

From (\ref{3}) and the assumption $F_{m}\cap Z\subset \overline{F}_{\alpha }$
when $\dim Z<\infty ,$ it follows that $F_{m}^{\ddagger }\subset \cup
_{\alpha <m}\overline{F}_{\alpha }.$ Since the sets $F_{\alpha }$ are
linearly ordered by inclusion and nowhere dense when $\alpha <m,$ $\cup
_{\alpha <m}\overline{F}_{\alpha }=\cup _{\alpha <m,\alpha \in \Bbb{Q}}%
\overline{F}_{\alpha }$ is of first category. Thus, $F_{m}^{\ddagger }$ has
empty interior. In addition, if $X$ is reflexive, then $\overline{F}_{\alpha
}$ is Haar null, so that $\cup _{\alpha <m,\alpha \in \Bbb{Q}}\overline{F}
_{\alpha }$ is Haar null.
\end{proof}

\begin{corollary}
\label{cor22}Let $X$ be a separable Banach space and let $f:X\rightarrow 
\Bbb{R}$ be quasiconvex and densely continuous. With $m:=\mathcal{T}\limfunc{
ess}\inf_{X}f,$ suppose that $F_{m}$ is contained in some proper affine
subspace of $X.$ Then, $f$ is continuous and G\^{a}teaux differentiable on a
dense subset of $X.$ \newline
If $X$ is reflexive and $F_{m}$ is contained in some proper Borel affine
subspace, $f$ is continuous and G\^{a}teaux differentiable on $X$ except at
the points of a Haar null set.
\end{corollary}

\begin{proof}
With no loss of generality, assume $\dim X=\infty $ (if $\dim X<\infty ,$
Crouzeix's theorem gives a stronger result). Let $Y$ denote a proper affine
subspace of $X$ containing $F_{m}.$ Evidently, $F_{m}^{\ddagger }\subset
Y^{\ddagger }$ and, as noted at the beginning of this section, $Y^{\ddagger
}=Y$ since $Y$ is affine. Since a proper affine subspace of $X$ has empty
interior, $F_{m}^{\ddagger }$ has empty interior, so that Theorem \ref{th20}
applies.

Since $Y\neq X,$ there is $\xi \in X\backslash \{0\}$ such that $\Bbb{R}\xi $
intersects no translate of $Y$ at more than one point. Hence, if $E\subset X$
is Borel and $\mu (E):=\lambda _{1}(\Bbb{R}\xi \cap E)$ where $\lambda _{1}$
is the Lebesgue measure on $\Bbb{R}\xi ,$ then $\mu (x+Y)=0$ for every $x\in
X$ provided that $Y$ is Borel. Thus, if $X$ is reflexive, Theorem \ref{th20}
ensures that $f$ is continuous and G\^{a}teaux differentiable on $X$ except
at the points of a Haar null set.
\end{proof}

An interesting outcome of Corollary \ref{cor22} is:

\begin{corollary}
\label{cor23}Let $X$ be a separable Banach space and let $f:X\rightarrow 
\Bbb{R}$ be even, quasiconvex and densely continuous. Then, $f$ is
continuous and G\^{a}teaux differentiable on a dense subset of $X.$ \newline
If $X$ is reflexive, $f$ is continuous and G\^{a}teaux differentiable on $X$
except at the points of a Haar null set.
\end{corollary}

\begin{proof}
If $m:=\mathcal{T}\limfunc{ess}\inf_{X}f=-\infty ,$ a stronger property is
proved in Corollary \ref{cor10}. Accordingly, we only give a proof when $%
m>-\infty $ and $F_{m}\neq \emptyset $ since Theorems \ref{th14} and \ref
{th15} yield a better result when, more generally, $F_{m}$ is nowhere dense.

Since $f$ is even, $F_{m}=-F_{m}$ and so $0\in F_{m}$ by the convexity of $%
F_{m}\neq \emptyset .$ As a result, the convex cone $CF_{m}:=\cup _{\lambda
>0}\lambda F_{m}$ is the countable union $CF_{m}=\cup _{n\in \Bbb{N}%
}\,nF_{m} $ (of first category) and $CF_{m}=-CF_{m},$ so that $CF_{m}$ is a
vector subspace of $X$ of first category. Thus, $CF_{m}\neq X.$ Since $%
F_{m}\subset CF_{m},$ Corollary \ref{cor22} ensures that $f$ is continuous
and G\^{a}teaux differentiable on a dense subset of $X.$

If $X$ is reflexive, the result also follows from Corollary \ref{cor22} if
we show that $F_{m}$ is contained in a Borel proper (vector) subspace of $X.$
This can be seen by a modification of the above argument. Specifically, $%
F_{\alpha }=-F_{\alpha }$ for every $\alpha \in \Bbb{R},$ not just $\alpha
=m,$ whence $\overline{F}_{\alpha }=-\overline{F}_{\alpha }$ and $C\overline{%
F}_{\alpha }:=\cup _{\lambda >0}\lambda \overline{F}_{\alpha }=\cup _{n\in 
\Bbb{N}}\,n\overline{F}_{\alpha }.$ Thus, $C\overline{F}_{\alpha }$ is a
Borel vector subspace of $X$ for every $\alpha \in \Bbb{R}$ such that $%
F_{\alpha }\neq \emptyset $ and $C\overline{F}_{\alpha }=\emptyset $
otherwise.

If $\alpha <m,$ then $F_{\alpha }$ is nowhere dense (Theorem \ref{th4}),
whence $C\overline{F}_{\alpha }$ is of first category. The subspaces $C%
\overline{F}_{\alpha }$ are linearly ordered by inclusion and so $\cup
_{\alpha <m}C\overline{F}_{\alpha }=\cup _{\alpha <m,\alpha \in \Bbb{Q}}C%
\overline{F}_{\alpha }$ is a Borel vector subspace of $X$ of first category (%
$F_{m}\neq \emptyset $ ensures that $F_{\alpha }\neq \emptyset $ for some $%
\alpha <m,$ so that $C\overline{F}_{\alpha }$ is a subspace). Obviously, $%
F_{m}=\cup _{\alpha <m}F_{\alpha }\subset \cup _{\alpha <m}\overline{F}%
_{\alpha }\subset \cup _{\alpha <m}C\overline{F}_{\alpha }.$ This completes
the proof.
\end{proof}

Below, we show that differentiability on a dense subset is still true when,
more generally, $f$ is invariant under the action of a linear periodic map $%
T $ with no fixed point. By this, we mean that $f(Tx)=f(x)$ where $T\in 
\mathcal{L}(X),$ $T^{p}=I$ for some integer $p>1$ and $1$ is not an
eigenvalue of $T,$ so that $Tx=x$ if and only if $x=0.$ In Corollary \ref
{cor23}, $p=2$ and $T=-I.$

\begin{corollary}
\label{cor24}Let $X$ be a separable Banach space and let $f:X\rightarrow 
\Bbb{R}$ be quasiconvex, densely continuous and invariant under the action
of a linear periodic map $T$ with no fixed point. Then, $f$ is continuous
and G\^{a}teaux differentiable on a dense subset of $X.$ \newline
If $X$ is reflexive, $f$ is continuous and G\^{a}teaux differentiable on $X$
except at the points of a Haar null set.
\end{corollary}

\begin{proof}
We only show how the proof of Corollary \ref{cor23} can be modified to yield
the desired result. The key point of that proof is that, when $m>-\infty $
and $F_{m}\neq \emptyset ,$ the cone $CF_{m}$ is both of first category and
a vector subspace of $X.$ That $CF_{m}$ is of first category is true
irrespective of any symmetry. In the proof of Corollary \ref{cor23}, that $%
CF_{m}$ is a vector subspace of $X$ is an immediate by-product of $%
F_{m}=-F_{m},$ but a routine check reveals that this remains true if, for
every $x\in F_{m},$ there is $\varepsilon >0$ (possibly depending upon $x$)
such that $-\varepsilon x\in F_{m}.$ The proof that this holds under the
hypotheses of the corollary is given below, with $F_{m}$ replaced by any
nonempty convex subset $E\subset X$ invariant under $T$ (i.e., $T(E)\subset
E $).

Since the existence of $\varepsilon $ is obvious if $x=0,$ assume $x\neq 0.$
With the notation introduced before the corollary, it is clear that $%
x,Tx,...,T^{p-1}x\in E.$ Furthermore, since $T^{p}=I,$ the point $%
y:=x+Tx+\cdots +T^{p-1}x$ is a fixed point of $T,$ (i.e., $Ty=y$) so that $%
y=0$ since $T$ has no other fixed point. As a result, $0=\frac{1}{p}x+\frac{1%
}{p}Tx+\cdots +\frac{1}{p}T^{p-1}x$ and hence $0\in C_{x}:=\limfunc{ conv}
\{x,...,T^{p-1}x\}\subset E.$

We claim that $C_{x}$ is a neighborhood of $0$ in $Y_{x}:=\limfunc{span}%
\{x,...,T^{p-1}x\}.$ Indeed, the linear mapping $\Lambda $ on $\Bbb{R}^{p}$
defined by $\mathbf{\lambda }:=(\lambda _{0},...,\lambda _{p-1})\in \Bbb{R}
^{p}\mapsto \Lambda \mathbf{\lambda }:=\Sigma _{j=0}^{p-1}\lambda
_{j}T^{j}x\in Y_{x}$ is surjective and $\ker \Lambda $ contains $p^{-1}%
\mathbf{1,}$ where $\mathbf{1}:=(1,...,1).$ Since the subspace $Z:=\left\{ 
\mathbf{\lambda }\in \Bbb{R}^{p}:\Sigma _{j=0}^{p-1}\lambda _{j}=0\right\} $
does not contain $p^{-1}\mathbf{1},$ it follows that $\Lambda :Z\rightarrow
Y_{x}$ is still surjective. Let $U\subset Z$ denote the neighborhood of $0$
defined by $U:=\{\mathbf{\lambda }\in Z:|\lambda _{j}|<p^{-1}\}.$ Since
linear surjective maps are open, $\Lambda (U)$ is an open neighborhood of $0$
in $Y_{x}.$ Now, $\Lambda (U)=\Lambda (U+p^{-1}\mathbf{1)}$ and, if $\mathbf{%
\ \ \ \ \ \mu }\in U+p^{-1}\mathbf{1,}$ then $\mathbf{\mu }=\mathbf{\lambda }%
+p^{-1}\mathbf{1}$ with $\mathbf{\ \lambda }\in U,$ so that $\mu _{j}>0$ and 
$\Sigma _{j=0}^{p-1}\mu _{j}=\Sigma _{j=0}^{p-1}\left( \lambda
_{j}+p^{-1}\right) =1.$ Thus, every vector in the open neighborhood $\Lambda
(U)$ of $0$ in $Y_{x}$ is a convex combination of $x,...,T^{p-1}x$ and
therefore in $C_{x}.$ This proves the claim.

Since $-x\in Y_{x}$ and every neighborhood of $0$ in $Y_{x}$ is absorbing,
it follows that $-\varepsilon x\in C_{x}\subset E$ for some $\varepsilon >0,$
as had to be proved.

The continuity of $T$ implies that $\overline{F}_{\alpha }$ (convex) is
invariant under $T$ for every $\alpha \in \Bbb{R}.$ Thus, from the above, $C%
\overline{F}_{\alpha }$ is a vector subspace of $X$ whenever $F_{\alpha
}\neq \emptyset .$ As a result, the proof of Corollary \ref{cor23} can also
be repeated when $X$ is reflexive.
\end{proof}

Corollaries \ref{cor23} and \ref{cor24} remain true when $f$ is replaced by
its lsc hull $\underset{-}{f}$ which, as is easily checked, exhibits the
same symmetry as $f.$ Also, if Theorem \ref{th20} is applicable to $%
\underset{-}{f},$ it is applicable to $f$ by Lemma \ref{lm18} (ii) since $%
F_{m}\subset G_{m}:=\{x\in X:\underset{-}{f}(x)<m\}$ is obvious from $%
\underset{-}{f}\leq f.$ However, there seems to be no reason why the
converse should be true (compare with Subsection \ref{lsc-hulls}).

\section{An example\label{example}}

We give an example of a function $f$ satisfying all the assumptions of
Corollaries \ref{cor21} and \ref{cor23} (hence of Corollaries \ref{cor22}
and \ref{cor24} as well, because Corollary \ref{cor23} is a special case of
both). This example was used in \cite{Ra12} to produce a nontrivial densely
continuous quasiconvex function. It will also show that the
differentiability properties of $f$ at the points of $F_{m}^{\prime }$ are
generally much weaker when $F_{m}$ is not nowhere dense than they are when $%
F_{m}$ is nowhere dense.

\begin{example}
\label{ex1}Let $X$ be an infinite dimensional separable Banach space and let 
$(x_{n})_{n\in \Bbb{N}}\subset X$ be a dense sequence. Set $L_{n}:=\limfunc{
span}\{x_{1},...,x_{n}\},$ so that $L:=\cup _{n\in \Bbb{N}}L_{n}$ is a dense
subspace of $X$ and $L$ is of first category since $\dim L_{n}<\infty .$
After passing to a subsequence, we may assume $L_{n}\varsubsetneq L_{n+1}$
without changing $L.$ Now, let $(\alpha _{n})\subset \Bbb{R}$ be a strictly
increasing sequence such that $\lim_{n\rightarrow \infty }\alpha _{n}=1.$
Set $f(x)=1$ if $x\in X\backslash L,f(x)=\alpha _{1}$ if $x\in L_{1}$ and $%
f(x)=\alpha _{n}$ if $x\in L_{n}\backslash L_{n-1},n\geq 2.$
\end{example}

In Example \ref{ex1}, $m:=\mathcal{T}\limfunc{ess}\inf_{X}f=1,F_{\alpha }=X$
if $\alpha >1,F_{1}=L$ is convex of first category but everywhere dense and $%
F_{1}^{\prime }=X.$ Also, if $\alpha _{1}<\alpha <1,$ there is a unique $%
n\in \Bbb{N}$ such that $\alpha _{n}<\alpha \leq \alpha _{n+1}$ and so $%
F_{\alpha }=L_{n}$ is (closed, convex and) nowhere dense. If $\alpha \leq
\alpha _{1},$ then $F_{\alpha }=\emptyset .$ In particular, $f$ is
quasiconvex.

By Theorem \ref{th4} (or because $f$ is lsc; see Remark \ref{rm2} below), $f$
is densely continuous. Actually, a direct verification shows that the set of
points of discontinuity of $f$ is exactly $L.$ If $x\notin L$ and $h\in
X\backslash \{0\},$ then $x+th\in L$ if and only if $x+th\in L_{n}$ for some 
$n.$ This can only happen if $t\neq 0$ and for at most one $t.$ Indeed, if $%
x+th\in L_{n_{1}}$ and $x+sh\in L_{n_{2}}$ with $s\neq t$ and (say) $%
n_{2}\geq n_{1},$ then $x+sh-x-th=(s-t)h\in L_{n_{2}},$ so that $h\in
L_{n_{2}}.$ But then, $sh\in L_{n_{2}}$ and so $x\in L_{n_{2}}\subset L,$
which is a contradiction.

The above property shows that if $x\notin L$ and $h\in X\backslash \{0\},$
there is $t_{h}>0$ such that $(x-t_{h}h,x+t_{h}h)\subset X\backslash L.$
Since $f=1$ on $X\backslash L,$ it follows that $f$ is also G\^{a}teaux
differentiable at every point of $X\backslash L$ (and $Df(x)=0$). Note that $%
L$ has countable dimension, so that it is Borel and Haar null regardless of
the reflexivity of $X.$

Now, it is plain that $f$ is even, so that it satisfies the hypotheses of
Corollary \ref{cor23}. To see that it also satisfies the hypotheses of
Corollary \ref{cor22}, let $Z$ be a finite dimensional subspace of $X$ and
let $x\in F_{1}\cap Z=L\cap Z$ (recall $m=1$). Obviously, $L\cap Z=\cup
_{n\in \Bbb{N}}(L_{n}\cap Z)$ and $L_{n}\cap Z$ is a nondecreasing sequence
of subspaces of $Z.$ Since $\dim Z<\infty ,$ it follows that $L_{n}\cap Z$
is independent of $n$ large enough, say $n\geq n_{0}$ and so $L\cap
Z=L_{n_{0}}\cap Z.$ Thus, $x\in L_{n_{0}}$ and so $f(x)\leq \alpha
_{n_{0}}<1=m.$

\begin{remark}
\label{rm2}In Example \ref{ex1}, $F_{\alpha }^{\prime }=X$ if $\alpha \geq %
1,F_{\alpha }^{\prime }=L_{n}$ if $\alpha _{1}\leq \alpha <1,$ where $n\in 
\Bbb{N}$ is the unique integer such that $\alpha _{n}\leq \alpha <\alpha
_{n+1}$ and $F_{\alpha }^{\prime }=\emptyset $ if $\alpha <\alpha _{1}.$
Thus, $F_{\alpha }^{\prime }$ is always closed, so that $f$ is lsc. Since $X$
and all its finite dimensional subspaces are ideally convex, this also shows
that $f$ is ideally quasiconvex (Definition \ref{def1}), but $F_{1}$ is
dense with empty interior, so that $F_{1}$ is not ideally quasiconvex; see
the comments before Theorem \ref{th7}. Thus, $f$ is not strongly ideally
quasiconvex.
\end{remark}

In general, when $m>-\infty $ and $F_{m}^{\prime }$ is of second category,
then $\overset{\circ }{F_{m}^{\prime }}\neq \emptyset $ is dense in $%
F_{m}^{\prime }$ (Theorem \ref{th5} (ii)). If also $F_{m}$ is nowhere dense,
then $f=m$ on the open set $\overset{\circ }{F_{m}^{\prime }}\backslash 
\overline{F}_{m}$ dense in $F_{m}^{\prime },$ so that $f$ is Fr\'{e}chet
differentiable with derivative $0$ on an open and dense subset of $%
F_{m}^{\prime }.$ This was used in the proof of Theorem \ref{th13}.

In the above example, $m=1$ and $F_{1}^{\prime }=X$ is of second category.
Since $F_{1}=L$ is dense in $X$ and $f$ is not G\^{a}teaux differentiable at
any point of $F_{1}$ (Theorem \ref{th11}), $f$ cannot be Fr\'{e}chet
differentiable on an open and dense subset of $F_{1}^{\prime }$ but, since
it is continuous and G\^{a}teaux differentiable on the residual subset $%
f^{-1}(1)=X\backslash L,$ it could conceivably be Fr\'{e}chet differentiable
on this set, or perhaps on a smaller but still residual subset. Below, we
show not only that this is not the case, but even that the set of points of $%
X\backslash L$ where $f$ is Hadamard differentiable is of first category.
This establishes that, in general, the differentiability properties at the
points of $f^{-1}(m)$ are indeed weaker when $F_{m}$ is not nowhere dense
than they are when $F_{m}$ is nowhere dense.

The notation of Example \ref{ex1} is used in the next lemma. Also, $%
d(x,L_{n})$ denotes the distance from $x\in X$ to the subspace $L_{n}.$

\begin{lemma}
\label{lm25}Let $(\beta _{n})_{n\in \Bbb{N}}\subset (0,\infty )$ be any
sequence. The set $W:=\{x\in X:d(x,L_{n})<\beta _{n}$ for infinitely many
indices $n\in \Bbb{N}\}$ is residual in $X.$
\end{lemma}

\begin{proof}
By the continuity of $d(\cdot ,L_{n}),$ the set $W_{n}:=\{x\in
X:d(x,L_{n})<\beta _{n}\}$ is open in $X$ for every $n\in \Bbb{N}.$ In
addition, $L_{k}\subset W_{n}$ for every $n\geq k,$ so that $L_{k}\subset
\cup _{n\geq j}W_{n}$ for every $j,k\in \Bbb{N}.$ Thus, $L\subset \cup
_{n\geq j}W_{n}$ for every $j\in \Bbb{N}$ so that $\cup _{n\geq j}W_{n}$ is
open and dense in $X$ irrespective of $j\in \Bbb{N}.$ As a result, $\cap
_{j\in \Bbb{N}}\cup _{n\geq j}W_{n}$ is residual and coincides with the set $%
W$ of the lemma.
\end{proof}

\begin{theorem}
\label{th26}The set of points where the function $f$ of Example \ref{ex1} is
Hadamard differentiable is of first category in $X.$
\end{theorem}

\begin{proof}
If $x\in X$ and $n\in \Bbb{N},$ let $\xi _{n}=\xi _{n}(x)\in L_{n}$ be such
that $||x-\xi _{n}||=d(x,L_{n})$ (recall $\dim L_{n}<\infty $). In Lemma \ref
{lm25}, let $\beta _{n}:=(1-\alpha _{n})^{2}$ (notation of Example \ref{ex1}%
), so that $W:=\{x\in X:||x-\xi _{n}||<(1-\alpha _{n})^{2}$ for infinitely
many indices $n\in \Bbb{N}\}$ is residual in $X.$ Below, we show that $f$ is
not Hadamard differentiable at any point $x\in W\backslash L.$ Since $f$ is
not even G\^{a}teaux differentiable at any point of $L=F_{1}$ by Theorem \ref
{th11} (i), this implies that $f$ is not Hadamard differentiable at any
point of $W,$ which proves the theorem.

If $x\in X\backslash L,$ then $\xi _{n}\neq x$ for every $n$ and $f(x)=1.$
We already know that $f$ is continuous and G\^{a}teaux differentiable at $x$
with derivative $0$ (the latter also follows from Theorem \ref{th11} (ii)).
To show that $f$ is not Hadamard differentiable at $x\in W\backslash L,$ it
suffices to find sequences $(h_{n})\subset X$ and $(t_{n})\subset (0,\infty
),$ both tending to $0,$ such that $\lim_{n\rightarrow \infty }\frac{
f(x+t_{n}h_{n})-f(x)}{t_{n}}\neq 0.$

Set $h_{n}:=\frac{\xi _{n}-x}{||x-\xi _{n}||^{1/2}}$ and $t_{n}:=||x-\xi
_{n}||^{1/2}.$ That $\lim_{n\rightarrow \infty }t_{n}=0$ and $%
\lim_{n\rightarrow \infty }h_{n}=0$ follows from the definition of $\xi _{n}$
and the denseness of $L$ in $X.$ On the other hand, 
\begin{equation*}
\left| \frac{f(x+t_{n}h_{n})-f(x)}{t_{n}}\right| =\frac{1-f(\xi _{n})}{
||x-\xi _{n}||^{1/2}}\geq \frac{1-\alpha _{n}}{||x-\xi _{n}||^{1/2}},
\end{equation*}
where the inequality follows from $\xi _{n}\in L_{n}$ and $f\leq \alpha _{n}$
on $L_{n}.$ Since $x\in W,$ there are infinitely many indices $n$ such that $%
||x-\xi _{n}||^{1/2}<1-\alpha _{n}.$ Therefore, $\frac{1-\alpha _{n}}{
||x-\xi _{n}||^{1/2}}>1$ for infinitely many indices $n.$ As a result, $%
\frac{f(x+t_{n}h_{n})-f(x)}{t_{n}}$ does not tend to $0$ and the proof is
complete.
\end{proof}

Although Theorem \ref{th26} demonstrates that the cases when $F_{m}$ is
nowhere dense or just of first category are different, it does not rule out
that $f$ is Hadamard differentiable except at the points of a Haar null set.
With special choices of $X,L_{n}$ and $\alpha _{n},$ it can be shown that
the set of points where $f$ is not Hadamard differentiable is not Aronszajn
null, but ``not Haar null'' has remained elusive.

\section{A peculiar convex subset of $\ell ^{2}$\label{peculiar}}

Corollaries \ref{cor23} and \ref{cor24} raise the question whether symmetry
or any extra assumption is really needed to ensure that a densely continuous
quasiconvex function is continuous and G\^{a}teaux differentiable at the
points of a dense subset of the separable Banach space $X.$ By Theorem \ref
{th20}, this would be true if $C^{\ddagger }$ in (\ref{3}) had empty
interior whenever $C$ is a convex subset of $X$ of first category.

In this section, we show that such a property does not hold, even when $%
X=\ell ^{2},$ by constructing a convex subset $C$ of $\ell ^{2}$ of first
category such that $C^{\ddagger }=\ell ^{2}$ (see (\ref{3})). Similar
subsets can be constructed in any $\ell ^{p},1\leq p<\infty .$

Let $P:=\{(x_{j}):x_{j}\geq 0$ $\forall j\in \Bbb{N}\}\subset \ell ^{2}$
denote the nonnegative cone, a closed convex subset with empty interior. Set 
$u:=(j^{-1})_{j\in \Bbb{N}}\in P$ and, for $n\in \Bbb{N},$ define $%
C_{n}:=-nu+P,$ an increasing sequence of closed convex subsets (not cones)
with empty interior. Then, $C:=\cup _{n\in \Bbb{N}}C_{n}$ is a convex subset
(even a cone) of $\ell ^{2}$ of first category and $P\subset C.$

\begin{lemma}
\label{lm27}For every $x\in \ell ^{2}\backslash \{0\},$ there is $h\in
P\backslash \{0\}$ such that, for every $t>0,x_{j}+th_{j}\geq 0$ for $j$
large enough.
\end{lemma}

\begin{proof}
Given $x\in \ell ^{2},$ choose a strictly increasing sequence $%
(n_{i})\subset \Bbb{N}$ such that $\Sigma _{j=n_{i}+1}^{\infty
}x_{j}^{2}\leq i^{-4}.$ For $j\in \Bbb{N},$ let $h_{j}:=|x_{j}|$ if $1\leq
j\leq n_{1}$ and $h_{j}:=i|x_{j}|$ if $n_{i}+1\leq j\leq n_{i+1}.$ That $%
h:=(h_{j})\in \ell ^{2}$ follows from 
\begin{multline*}
\Sigma _{j=1}^{\infty }h_{j}^{2}=\Sigma _{j=1}^{n_{1}}h_{j}^{2}+\Sigma
_{i=1}^{\infty }\left( \Sigma _{j=n_{i}+1}^{n_{i+1}}h_{j}^{2}\right) = \\
\Sigma _{j=1}^{n_{1}}x_{j}^{2}+\Sigma _{i=1}^{\infty }i^{2}\left( \Sigma
_{j=n_{i}+1}^{n_{i+1}}x_{j}^{2}\right) \leq \Sigma
_{j=1}^{n_{1}}x_{j}^{2}+\Sigma _{i=1}^{\infty }i^{2}\left( \Sigma
_{j=n_{i}+1}^{\infty }x_{j}^{2}\right) \leq \\
\Sigma _{j=1}^{n_{1}}x_{j}^{2}+\Sigma _{i=1}^{\infty }i^{-2}<\infty .
\end{multline*}
It is obvious that $h\in P\backslash \{0\}$ and, if $t>0$ and $i\geq t^{-1},$
then, $x_{j}+th_{j}\geq 0$ provided that $j\geq n_{i}+1.$
\end{proof}

If $x\in \ell ^{2}\backslash \{0\},$ let $h\in P\backslash \{0\}$ be given
by Lemma \ref{lm27}, so that if $t>0,$ then $x_{j}+th_{j}\geq 0$ for $j$
large enough. Hence, $x_{j}+th_{j}+nu_{j}>0$ irrespective of $n\in \Bbb{N}$
for the same indices $j$ since $u_{j}>0$ for every $j.$ In addition, the
finitely many remaining terms $x_{j}+th_{j}+nu_{j}$ can all be made positive
if $n$ is chosen large enough. Thus, $x+th+nu\in P$ if $n$ is large enough,
whence $x+th\in C_{n}\subset C$ for every $t>0.$

Since $t>0$ above may be arbitrarily small, there is a sequence of $\limfunc{
span}\{x,h\}\cap C$ tending to $x,$ whence $x\in C^{\ddagger }.$ Thus, $\ell
^{2}\backslash \{0\}\subset C^{\ddagger }.$ Since also $0\in P\subset
C\subset C^{\ddagger },$ the proof that $C^{\ddagger }=\ell ^{2}$ is
complete.

Anecdotally, that $C^{\ddagger }=\ell ^{2}$ implies that $C$ has another
uncommon feature. Since it is of first category, $C\neq \ell ^{2}$ and so,
by the so-called Stone-Kakutani property, $C$ is contained in a convex
subset $\widehat{C}$ such that $\ell ^{2}\backslash \widehat{C}$ is also
convex and nonempty (\cite[p. 71]{Sc86}). But $\widehat{C}$ cannot be a
half-space $\Sigma :=\{x\in \ell ^{2}:l(x)\leq a\}$ (let alone $\{x\in \ell
^{2}:l(x)<a\}$) for any $a\in \Bbb{R}$ and any linear form $l$ on $\ell
^{2}, $\emph{\ continuous or not}. Indeed, if $C\subset \Sigma ,$ then $%
C^{\ddagger }\subset \Sigma ^{\ddagger }.$ Since every linear form is
continuous on finite dimensional subspaces, if follows from (\ref{3}) that $%
\Sigma ^{\ddagger }=\Sigma .$ Since $\Sigma \neq \ell ^{2}$ is obvious, a
contradiction is reached with $C^{\ddagger }=\ell ^{2}.$

The procedure used in Example \ref{ex1} shows how to construct quasiconvex
functions $f$ with $F_{m}=C$ which, in addition, are continuous at every
point of $X\backslash C$ (and so, densely continuous). Indeed, with the same
increasing sequence $(\alpha _{n}),$ it suffices to define $f(x)=\alpha _{1}$
on $C_{1},f(x)=\alpha _{n}$ on $C_{n}\backslash C_{n-1},n\geq 2$ and $f(x)=1$
on $X\backslash C.$ If $x\notin C,$ then $x\notin C_{n}$ for every $n$ and
so there is an open ball $B(x,\varepsilon _{n})$ such that $B(x,\varepsilon
_{n})\cap C_{n}=\emptyset .$ Thus, if $y\in B(x,\varepsilon _{n}),$ then
either $f(y)=1$ (if $y\notin C$) or $f(y)\geq \alpha _{n+1}$ (if $y\in C_{j}$
with $j\geq n+1$). In either case, $f(y)$ is arbitrarily close to $1$ if $n$
is large enough, which proves the continuity of $f$ at $x.$ Unfortunately,
the G\^{a}teaux differentiability question seems difficult to settle at the
points $x\notin C,$ irrespective of the choice of the sequence $(\alpha
_{n}).$ (The characteristic function of $X\backslash C$ is nowhere
G\^{a}teaux differentiable, but since it is also nowhere continuous, it is
not a counter example.)

\section{Essential G\^{a}teaux differentiability on a dense subset\label%
{essential}}

In the previous section, we saw that the question whether every densely
continuous quasiconvex function on a separable Banach space $X$ is
continuous and G\^{a}teaux differentiable on a dense subset of $X$ is not
fully resolved by Theorem \ref{th20}. We complete this paper by showing that
a positive answer can be given if the conditions for differentiability are
suitably relaxed.

The classical concept of G\^{a}teaux differentiability of $f$ at a point $x$
requires $\lim_{t\rightarrow 0}\frac{f(x+th)-f(x)}{t}=l_{x}(h)$ for some $%
l_{x}\in X^{*}$ and every $h\in X\backslash \{0\},$ that is, for every $h\in
S_{X},$ the unit sphere of $X.$ If this holds only for $h$ in a dense subset 
$\Omega $ of $S_{X},$ then $l_{x}$ is uniquely determined by $\Omega ,$ but
since different subsets $\Omega $ may produce different linear forms $l_{x},$
this does not provide a sound generalization of G\^{a}teaux
differentiability. On the other hand, if $\Omega $ is residual in $S_{X}$
and if $l_{x}^{\prime }\in X^{*}$ is similarly defined when $\Omega $ is
replaced by another residual set of directions $\Omega ^{\prime },$ then $%
l_{x}^{\prime }=l_{x}$ since $\Omega \cap \Omega ^{\prime }$ is residual,
and therefore dense, in $S_{X}$ (since $X$ is a Banach space, $S_{X}$ is a
complete metric space and so a Baire space).

Accordingly, we shall say that $f$ is \textit{essentially G\^{a}teaux
differentiable} at $x$ if $l_{x}(h)=\lim_{t\rightarrow 0}\frac{f(x+th)-f(x)}{%
t}$ for some $l_{x}\in X^{*}$ and every $h$ in some residual subset\footnote{%
By the corollary of the Kuratowski-Ulam theorem used in the proof of Theorem 
\ref{th29}, the definition is unchanged if it is required that $h$ belongs
to a residual cone in $X\backslash \{0\}.$} $\Omega $ of $S_{X}.$ If so,
from the above remarks, the (essential) G\^{a}teaux derivative $Df(x):=l_{x}$
is well defined and independent of $\Omega .$ The next theorem states that
for densely continuous quasiconvex functions on a separable Banach space,
continuity plus essential G\^{a}teaux differentiability always hold on a
dense subset. For clarity, we expound the simple geometric property of
convex sets used in the proof.

\begin{lemma}
\label{lm28}Let $X$ be a real vector space and let $C\varsubsetneq X$ be a
convex subset. Given $x\in X\backslash C$ and $h\in X\backslash \{0\},$
there is $t_{h}>0$ such that either $(x,x+t_{h}h)\subset C$ or $%
[x,x+t_{h}h)\subset X\backslash C.$
\end{lemma}

\begin{proof}
Suppose that there is no $t_{h}>0$ such that $(x,x+t_{h}h)\subset C,$ i.e.,
that $(x,x+th)$ contains a point of $X\backslash C$ for every $t>0.$ We show
that there is $t_{h}>0$ such that $[x,x+th)\subset X\backslash C.$ If not,
since $x\notin C,$ there is a decreasing sequence $(t_{n})\subset (0,\infty
) $ such that $\lim t_{n}=0$ and $x+t_{n}h\in C$ for every $n.$ Since $C$ is
convex, $x+th\in C$ for $t\in [t_{n},t_{1}]$ and every $n.$ Thus, $%
(x,x+t_{1}h]\subset C$ contains no point of $X\backslash C,$ which is a
contradiction.
\end{proof}

\begin{theorem}
\label{th29}Let $X$ be a separable Banach space and let $f:X\rightarrow \Bbb{%
R}$ be quasiconvex and densely continuous. Then, $f$ is continuous and
essentially G\^{a}teaux differentiable on a dense subset of $X.$ \newline
If $X$ is reflexive, $f$ is continuous and essentially G\^{a}teaux
differentiable on $X$ except at the points of a Haar null set.
\end{theorem}

\begin{proof}
As in the proof of Theorem \ref{th20}, we may assume that $m:=\mathcal{T}%
\limfunc{ess}\inf_{X}f>-\infty $ and that $F_{m}^{\prime }$ is of second
category, whence $\overset{\circ }{F_{m}^{\prime }}\neq \emptyset $ by
Theorem \ref{th5} (ii). We shall show that if $x\in \overset{\circ }{%
F_{m}^{\prime }}\backslash F_{m},$ the essential G\^{a}teaux derivative of $%
f $ at $x$ exists and is $0.$ Since $\overset{\circ }{F_{m}^{\prime }}%
\backslash F_{m}$ is residual in $\overset{\circ }{F_{m}^{\prime }}$ and $f$
is continuous at the points of a residual subset of $\overset{\circ }{%
F_{m}^{\prime }},$ this implies that $f$ is continuous and essentially
G\^{a}teaux differentiable at the points of a residual subset of $\overset{%
\circ }{F_{m}^{\prime }}$ and hence at the points of a dense subset of $%
F_{m}^{\prime }.$ By Theorem \ref{th9} and Theorem \ref{th5} (iii), $f$ is
Hadamard differentiable at the points of a dense subset of $X\backslash
F_{m}^{\prime },$ which proves the first part of the theorem.

Let then $x\in \overset{\circ }{F_{m}^{\prime }}\backslash F_{m}$ be given.
The set $H_{x}:=\{h\in X\backslash \{0\}:(x,x+t_{h}h)\subset F_{m}$ for some 
$t_{h}>0\}$ is a cone\footnote{%
It is also convex, but this is unimportant in this proof.} (invariant by
multiplication by positive scalars). Also, $H_{x}\subset X\backslash \{0\}$
and, if $h\in H_{x},$ then $h\in n(F_{m}-x)$ for any $n\in \Bbb{N}$ such
that $n^{-1}<t_{h}.$ Since $F_{m}$ is of first category, $H_{x}\subset \cup
_{n\in \Bbb{N}}\,n(F_{m}-x)$ is of first category in $X$ and, hence, in $%
X\backslash \{0\}$ as well (if $E$ is a closed subset of $X$ with empty
interior in $X,$ then $E\backslash \{0\}$ is a closed subset of $X\backslash
\{0\}$ with empty interior in $X\backslash \{0\}$).

That $H_{x}\subset X\backslash \{0\}$ is a cone implies that it is
homeomorphic to $(0,\infty )\times (H_{x}\cap S_{X})$ through the
homeomorphism $z\rightarrow \left( ||z||,z||z||^{-1}\right) $ of $%
X\backslash \{0\}$ onto $(0,\infty )\times S_{X}.$ Since homeomorphisms
preserve Baire category, $(0,\infty )\times (H_{x}\cap S_{X})$ is of first
category in $(0,\infty )\times S_{X}.$ By a classical corollary of the
Kuratowski-Ulam theorem \cite[Theorem 15.3]{Ox80}, it follows that $%
H_{x}\cap S_{X}$ is of first category in $S_{X}.$ (The use of this theorem
requires the topology of $S_{X}$ to have a countable base, which is the case
since $X$ is separable.) Equivalently, $S_{X}\backslash H_{x}$ is residual
in $S_{X}.$

Now, if $h\in S_{X}\backslash H_{x},$ then $[x,x+t_{h}h)\subset X\backslash
F_{m}$ for some $t_{h}>0$ by Lemma \ref{lm28}. Thus, if $h\in
S_{X}\backslash (H_{x}\cup -H_{x})$ (also residual in $S_{X}$), there is $%
t_{h}>0$ such that $(x-t_{h}h,x+t_{h}h)\subset X\backslash F_{m}.$ Since $%
x\in \overset{\circ }{F_{m}^{\prime }},$ it is not restrictive to assume,
after shrinking $t_{h}$ if necessary, that $(x-t_{h}h,x+t_{h}h)\subset 
\overset{\circ }{F_{m}^{\prime }}\backslash F_{m}\subset f^{-1}(m).$ This
makes it obvious that the derivative of $f$ in the direction $h$ exists and
is $0.$

Suppose now that $X$ is reflexive. From the above and Theorem \ref{th9}, the
points where $f$ does not have an essential G\^{a}teaux derivative can only
be in $F_{m},$ or in $\partial F_{m}^{\prime }$ (Aronszajn null since $%
\overset{\circ }{F_{m}^{\prime }}\neq \emptyset $) or in a subset of $%
X\backslash F_{m}^{\prime }$ contained in an Aronszajn null set. Thus, to
show that this set is contained in a Haar null set, it suffices to check
that $F_{m}$ is contained in a Haar null set. This follows from $%
F_{m}\subset \cup _{\alpha <m,\alpha \in \Bbb{Q}}\overline{F}_{\alpha }$
since $\overline{F}_{\alpha }$ is convex with empty interior when $\alpha <m$
(Theorem \ref{th4}) and $X$ is reflexive (Subsection \ref{haar}).

Since the set of points of discontinuity of $f$ is Haar null (Theorem \ref
{th6}) and the union of two Haar null sets is Haar null, the proof is
complete.
\end{proof}

There is a close connection between Theorem \ref{th20} and Theorem \ref{th29}
: If $H_{x}=\emptyset $ in the above proof, then $f$ is G\^{a}teaux
differentiable at $x$ (and $Df(x)=0$). On the other hand, by Lemma \ref{lm19}
, $H_{x}=\emptyset $ if and only if $x\notin F_{m}^{\ddagger }$.

\end{document}